\documentclass[11pt]{elsarticle}

\usepackage{lineno,hyperref}
\modulolinenumbers[5]
\usepackage[all]{xy}
\usepackage{tabularx}

\usepackage{amsmath,amssymb}
\usepackage{amsthm}
\usepackage{latexsym}
\usepackage{cleveref}
\usepackage{enumerate}
\usepackage{bm}
\usepackage{multirow}
\usepackage{geometry}
\usepackage{indentfirst}
\usepackage{fancyhdr}
\usepackage{graphicx}
\usepackage{subfigure}
\usepackage[noend]{algpseudocode}
\usepackage{algorithmicx,algorithm}
\usepackage{threeparttable}
\usepackage{booktabs}
\usepackage{makecell}
\usepackage{color}
\usepackage{float}
\usepackage{url}

\newtheoremstyle{mystyle}
{3pt}
{3pt}
{\rm}
{}
{\bfseries}
{.}
{.5em}
{}
\theoremstyle{plain}
\newcommand{\RNum}[1]{\uppercase\expandafter{\romannumeral #1\relax}}
\newtheorem{prop}{Proposition}[section]

\newtheorem{proposition}{Proposition}[section]
\newtheorem{corollary}{Corollary}[section]
\newtheorem{definition}{Definition}[section]
\newtheorem{remark}{Remark}[section]
\newtheorem{theorem}{Theorem}
\theoremstyle{definition}
\newtheorem{example}{Example}

\begin{document}
\begin{frontmatter}
\title{Oriented Supersingular Elliptic Curves and Eichler Orders of Prime Level}

\author[HUBU,NUDT]{Guanju Xiao}
\ead{gjxiao@amss.ac.cn}

\author[NUDT]{Zijian Zhou}
\ead{zhouzijian122006@163.com}

\author[NUDT]{Longjiang Qu\corref{mycorrespondingauthor}}
\cortext[mycorrespondingauthor]{Corresponding author}
\ead{ljqu_happy@hotmail.com}

\address[HUBU]{School of Cyber Science and Technology, Hubei University, Wuhan 430062, China}

\address[NUDT]{College of Science, National University of Defense Technology, Changsha 410073, China}

\begin{abstract}
Let $p>3$ be a prime and $E$ be a supersingular elliptic curve defined over $\mathbb{F}_{p^2}$. Let $c$ be a prime with $c < 3p/16$ and $G$ be a subgroup of $E[c]$ of order $c$. The pair $(E,G)$ is called a supersingular elliptic curve with level-$c$ structure, and the endomorphism ring $\text{End}(E,G)$ is isomorphic to an Eichler order with level $c$. We construct two kinds of Eichler orders $\mathcal{O}_c(q,r)$ and $\mathcal{O}'_c(q,r')$ with level $c$. Interestingly, we prove that each $\mathcal{O}_c(q,r)$ or $\mathcal{O}'_c(q,r')$ can represent a primitive reduced binary quadratic form with discriminant $-16cp$ or $-cp$ respectively. If a curve $E$ is $\mathbb{Z}[\sqrt{-cp}]$-oriented or $\mathbb{Z}[\frac{1+\sqrt{-cp}}{2}]$-oriented, then we prove that $\text{End}(E,G)$ is isomorphic to $\mathcal{O}_c(q,r)$ or $\mathcal{O}'_c(q,r')$ respectively. Due to the fact that $\mathbb{Z}[\sqrt{-cp}]$-oriented isogenies between $\mathbb{Z}[\sqrt{-cp}]$-oriented elliptic curves could be represented by quadratic forms, we show that these isogenies are reflected in the corresponding Eichler orders via the composition law for their corresponding quadratic forms.
\end{abstract}

\end{frontmatter}
\section{Introduction}
Let $p$ be a prime and $E$ be a supersingular elliptic curve over $\mathbb{F}_{p^2}$. The endomorphism ring $\text{End}(E)$ of $E$ is isomorphic to a maximal order in the quaternion algebra $B_{p,\infty}$ ramified exactly at $p$ and $\infty$. Moreover, Deuring \cite{MR5125} proved that there is a one-to-one correspondence between type classes of maximal orders in $B_{p,\infty}$ and isomorphism classes of supersingular elliptic curves up to the action of $\text{Gal}(\mathbb{F}_{p^2}/\mathbb{F}_p)$.

A special case of interest is when $E$ is defined over $\mathbb{F}_p$, in which case $\text{End}(E)$ contains an element $\pi$ such that $\pi ^2=-p$. It follows that the order $\mathbb{Z}[\sqrt{-p}]$ can be embedded into $\text{End}(E)$. Let $q \equiv 3 \pmod 8$ be a prime and $\left( \frac{p}{q} \right)=-1$. Ibukiyama\cite{MR683249} has given two kinds of maximal orders $\mathcal{O}(q,r)$ and $\mathcal{O}'(q,r')$ where $r$ (resp. $r'$) satisfies $r^2 +p \equiv 0 \pmod q$ (resp. $r'^2 +p \equiv 0 \pmod {4q}$), and proved that they are isomorphic to the endomorphism rings of some supersingular elliptic curves which are defined over $\mathbb{F}_p$. We \cite{xiao2022endomorphism} have established a one-to-one correspondence between $\mathbb{F}_p$-isomorphism classes of supersingular elliptic curves defined over $\mathbb{F}_p$ and primitive reduced binary quadratic forms with discriminant $-p$ or $-16p$, and also studied the connections of this correspondence and isogenies between supersingular elliptic curves defined over $\mathbb{F}_p$. Note that any supersingular elliptic curve defined over $\mathbb{F}_p$ has $\mathbb{Z}[\sqrt{-p}]$ as a suborder in its endomorphism ring. It is interesting to consider supersingular elliptic curves whose endomorphism rings contain quadratic order $\mathbb{Z}[\sqrt{-cp}]$ as a subring with $c>1$.

Let $c< 3p/16$ be a prime. Let $E$ be a supersingular elliptic curve defined over $\mathbb{F}_{p^2}$ and $G$ a subgroup of $E[c]$ with order $c$. The pair $(E,G)$ is called a supersingular elliptic curve with level-$c$ structure. Define $\text{End}(E,G)=\mathcal{O}(E, G):=\{\alpha \in \operatorname{End}(E): \alpha(G) \subseteq G\}\subseteq \text{End}(E)$. Just as the endomorphism ring $\text{End}(E)$ is isomorphic to a maximal order in $B_{p,\infty}$, Arpin \cite{Arpin2203} proved that $\text{End}(E,G)$ is isomorphic to an Eichler order with level $c$.

An interesting case is when $E/G \cong E^p$, in which case there exists a $c$-isogeny $\phi$ from $E$ to its Galois conjugate $E^p$ with $\ker(\phi)=G$. The pair $(E,\phi)$ is called a $(c,\epsilon)$-structure with $\epsilon \in \{\pm 1\}$ \cite{chenu2021higher}. In this case, $\text{End}(E)$ contains an element $\mu$ such that $\mu^2=-cp$, so these curves are also called $\mathbb{Z}[\sqrt{-cp}]$-oriented \cite{arpin2022orientations}. These curves are interesting because they are a generalization of curves defined over $\mathbb{F}_p$.

In this paper, we study endomorphism rings of these $\mathbb{Z}[\sqrt{-cp}]$-oriented supersingular elliptic curves. We define $q$ as (\ref{e1}) in Section 3. We give two kinds of Eichler orders $\mathcal{O}_c(q,r)$ and $\mathcal{O}'_c(q,r')$ where $r$ (resp. $r'$) satisfies $r^2 +cp \equiv 0 \pmod q$ (resp. $r'^2 +cp \equiv 0 \pmod {4q}$).

We study the relation between Eichler orders $\mathcal{O}_c(q,r)$'s and reduced binary quadratic forms with discriminant $-16cp$ in the genus class $\Lambda(q)$ (defined in Section 3). We prove this result by representing $\mathcal{O}_c(q,r) \setminus \mathbb{Z}[\sqrt{-cp}]$ with a quadratic form which is equivalent to the form $(q,4r,\frac{4r^2+4cp}{q})$ or $(q,-4r,\frac{4r^2+4cp}{q})$. On the contrary, given a reduced form $(a,b,a')$ with discriminant $-16cp$ in the genus class $\Lambda(q)$, we can compute a pair of supersingular $j$-invariants $j$ and $j^p$. Moreover, we prove that $j=j^p \in \mathbb{F}_p$ if and only if there exist $x,y \in \mathbb{Z}$ such that $x^2+ay^2=4c$. Similar results also hold for $\mathcal{O}'_c(q,r')$'s, while the corresponding reduced forms are of discriminant $-cp$ in the genus class $\Lambda' (q)$ (defined in Section 3). We prove that the endomorphism ring of any $\mathbb{Z}[\sqrt{-cp}]$-oriented supersingular elliptic curve with level-$c$ structure is isomorphic to some Eichler order $\mathcal{O}_c(q,r)$.

We also study the operations of $\mathbb{Z}[\sqrt{-cp}]$-oriented isogenies between $\mathbb{Z}[\sqrt{-cp}]$-oriented supersingular elliptic curves. Let $\varphi:E_1 \to E_2$ be a $\mathbb{Z}[\sqrt{-cp}]$-oriented isogeny, where $\text{deg}(\varphi)=\ell \neq c$ is a prime. As we know, the isogeny $\varphi$ can be represented by an ideal in $\mathbb{Z}[\sqrt{-cp}]$. Without loss of generality, we can denote it by a quadratic form $g$ with discriminant $-16cp$. If $\text{End}(E_1,G) \cong \mathcal{O}_c(q_1,r_1)$ where $q_1$ can be represented by a quadratic form $f$, then we have $\text{End}(E_2,\varphi(G)) \cong \mathcal{O}_c(q_2,r_2)$ where $q_2$ can be represented by a quadratic form $fg^2$. This shows an explicit relation between $\mathbb{Z}[\sqrt{-cp}]$-oriented isogenies and endomorphism rings of $\mathbb{Z}[\sqrt{-cp}]$-oriented supersingular elliptic curves.

The remainder of this paper is organized as follows. In Section 2, we review some preliminaries on $\mathbb{Z}[\sqrt{-cp}]$-oriented supersingular elliptic curves and binary quadratic forms. We define Eichler orders $\mathcal{O}_c(q,r)$ and $\mathcal{O}'_c(q,r')$ and discuss their isomorphism classes in Section 3. We prove the correspondence between $\mathbb{Z}[\sqrt{-cp}]$-oriented supersingular elliptic curves and binary quadratic forms in Section 4. In Section 5, we prove that the composition of forms corresponds to the action of $\mathbb{Z}[\sqrt{-cp}]$-oriented isogenies on $\mathbb{Z}[\sqrt{-cp}]$-oriented elliptic curves. Finally, we make a conclusion in Section 6.

\section{Preliminaries}

\subsection{Elliptic curves and isogenies}
We present some basic facts about elliptic curves over finite fields, and the readers can refer to \cite{MR2514094} for more details.

Let $\mathbb{F}_{p^k}$ be a finite field with characteristic $p>3$ and $k \in \mathbb{Z}_{\ge 1}$. An elliptic curve $E$ defined over $\mathbb{F}_{p^k}$ can be written as a Weierstrass model $E:Y^2=X^3+aX+b$ with $a,b \in \mathbb{F}_{p^k}$ and $4a^3+27b^2 \neq 0$. The $j$-invariant of $E$ is $j(E)=1728\cdot 4a^3/(4a^3+27b^2)$. Different elliptic curves with the same $j$-invariant are isomorphic over $\overline{\mathbb{F}}_{p^k}$. The chord-and-tangent addition law makes of $E(\mathbb{F}_{p^k})=\left\{(x,y)\in \mathbb{F}_{p^k}^2:y^2=x^3+ax+b\right\} \cup \left\{\infty \right\}$ an abelian group. For any integer $n\geq 2$ with $p\nmid n$, let $E[n]$ be the group of $n$-torsion points on $E$. We have $E[n] \cong \mathbb{Z}/n\mathbb{Z} \oplus \mathbb{Z}/n\mathbb{Z}$.

Let $E_1$ and $E_2$ be elliptic curves defined over $\mathbb{F}_{p^k}$. An isogeny from $E_1$ to $E_2$ is a morphism $\phi:E_1 \rightarrow E_2$ satisfying $\phi(\infty)=\infty$. If $\phi(E_1)=\{ \infty \}$, then we say $\phi=0$. If $\phi \neq 0$, then $\phi$ is a surjective group homomorphism with finite kernel.

An endomorphism of $E$ is an isogeny from $E$ to itself. The Frobenius map $\pi :(x,y) \mapsto (x^{p^k},y^{p^k})$ is an inseparable endomorphism. The characteristic polynomial of $\pi$ is $x^2-tx+p^k$, where $t$ is the trace of $\pi$. It is well known that $E$ is supersingular (resp. ordinary) if $p\mid t$ (resp. $p \nmid t$). Moreover, the $j$-invariant of every supersingular elliptic curve over $\overline{\mathbb{F}}_p$ is proved to be in $\mathbb{F}_{p^2}$ and it is called a supersingular $j$-invariant.

Each elliptic curve $E$ over $\mathbb{F}_{p^2}$ has a Galois-conjugate curve $E^p$, defined by $p$-powering all of the coefficients in the defining equation of $E$. The curve and its conjugate are connected by inseparable Frobenius $p$-isogenies $\pi_p :E \to E^p$ and $\pi_p:E^p \to E$, defined by $p$-th powering the coordinates (abusing notation, all inseparable $p$-isogenies will be denoted by $\pi_p$). Observe that $(E^p)^p=E$, and the composition of $\pi_p:E\to E^p$ and $\pi_p:E^p \to E$ is the $p^2$-power Frobenius endomorphism of $E$.

For a supersingular elliptic curve $E$ over $\mathbb{F}_{p^2}$, the endomorphism ring $\text{End}(E)$ is isomorphic to a maximal order of $B_{p,\infty}$, where $B_{p,\infty}$ is a quaternion algebra over $\mathbb{Q}$ ramified only at $p$ and $\infty$.

Deuring \cite{MR5125} gave an equivalence of categories between the supersingular $j$-invariants and the maximal orders in the quaternion algebra $B_{p,\infty}$. If $E$ is defined over $\mathbb{F}_{p^2} \setminus \mathbb{F}_p$, then $E \not \cong E^p$, but $\text{End}(E) \cong \text{End}(E^p)$ as maximal orders of $B_{p,\infty}$. If $E$ is defined over $\mathbb{F}_p$, then $E \cong E^p$ and $\text{End}(E)$ is the maximal order uniquely identified with the isomorphism class of $E$. Furthermore, if $E$ is a supersingular elliptic curve with $\text{End}(E)=\mathcal{O}$, then there is a one-to-one correspondence between isogenies $\phi:E\to E^{'}$ and left $\mathcal{O}$-ideals $I$. More details on the correspondence can be found in \cite[Chapter 42]{MR4279905}.

\subsection{Quaternion orders and ideals}
We recall some basic facts about quaternion orders in $B_{p,\infty}$ and their ideals. The general references are \cite{MR4279905, MR0579066}.

For $a,b \in \mathbb{Q}^{\times}$, we denote by $H(a,b)=\mathbb{Q}+\alpha \mathbb{Q}+\beta \mathbb{Q}+\alpha \beta \mathbb{Q}$ the quaternion algebra over $\mathbb{Q}$ with basis $1, \alpha , \beta , \alpha \beta$ such that $\alpha^2=a$, $\beta^2=b$ and $\alpha \beta =-\beta \alpha$. The Hilbert symbol $(a,b)_{\ell}=-1$ if and only if $H(a,b)$ ramifies at $\ell$, where $\ell$ is a finite prime or $\infty$.

 We are interested in $B_{p,\infty}$, the unique quaternion algebra (up to isomorphism) ramified exactly at $p$ and $\infty$, since the endomorphism ring of a supersingular elliptic curve over $\mathbb{F}_{p^2}$ is isomorphic to a maximal order of $B_{p,\infty}$.

Every quaternion algebra has a canonical involution that sends an element $\gamma=a_1+a_2\alpha +a_3\beta +a_4 \alpha \beta$ to its conjugate $\bar{\gamma}=a_1-a_2\alpha -a_3\beta -a_4 \alpha \beta$.
The reduced trace Trd and the reduced norm Nrd of $\gamma \in B_{p,\infty}$ are defined as:
$$\text{Trd}(\gamma)=\gamma + \bar{\gamma}, \quad \text{Nrd}(\gamma)=\gamma \bar{\gamma}.$$

An order $\mathcal{O}$ of $B_{p,\infty}$ is a subring of $B_{p,\infty}$ which is also a $\mathbb{Z}$-lattice of rank 4. Elements of an order $\mathcal{O}$ are said to be integral, since they have reduced trace and reduced norm in $\mathbb{Z}$. Given a basis $\langle \alpha_1,\alpha_2,\alpha_3,\alpha_4 \rangle$, the discriminant of $\mathcal{O}$ is defined as $\text{disc}(\mathcal{O})=\text{det}(\text{Trd}(\alpha_i \bar{\alpha}_j))_{i,j \in \{1,2,3,4\}}$. We have that the discriminant of an order is an integer and is independent of a choice of basis.

Two orders $\mathcal{O}_1$ and $\mathcal{O}_2$ are in the same type if and only if there exists an element $\alpha \in B_{p,\infty}^{\times}$ such that $\mathcal{O}_1=\alpha^{-1} \mathcal{O}_2 \alpha$. An order is called maximal if it is not properly contained in any other order. The discriminant of any maximal order in $B_{p,\infty}$ is $p^2$. A suborder $\mathcal{O}'$ of $\mathcal{O}$ is an order of rank 4 contained in $\mathcal{O}$. If $c=[\mathcal{O}:\mathcal{O}']$, then the discriminant of $\mathcal{O}'$ satisfies $\text{disc}(\mathcal{O}')=c^2 \text{disc}(\mathcal{O})$.

An Eichler order in a quaternion algebra is the intersection of two (not necessarily distinct) maximal orders. The level of an Eichler order in $B_{p,\infty}$ is given by its index in one of the maximal orders whose intersection defines the order (the index will be the same for either order).

Let $\mathcal{O}$ be an Eichler order of level $c$ and $I$ an invertible left ideal of $\mathcal{O}$.
Define the left order $\mathcal{O}_L(I)$ and the right order $\mathcal{O}_R(I)$ of $I$ by
$$\mathcal{O}_L(I)=\left\{x\in B_{p,\infty}:xI\subseteq I \right \}, \quad \mathcal{O}_R(I)=\left\{x\in B_{p,\infty}:Ix \subseteq I\right \}  .$$
Moreover, $\mathcal{O}_L(I)=\mathcal{O}$, and $\mathcal{O}_R(I)=\mathcal{O}^{'}$ is also an Eichler order of level $c$, in which case we say that $I$ connects $\mathcal{O}$ and $\mathcal{O}^{'}$. The inverse of an ideal is given by $I^{-1}=\{ a \in B_{p,\infty} : IaI \subseteq I\}$. In particular, we have $II^{-1}=\mathcal{O}_{L}(I)$ and $I^{-1}I=\mathcal{O}_R(I)$.

Fix an Eichler order $\mathcal{O}$ of $B_{p,\infty}$. Given a quadratic order $O$, we say that $O$ is optimally embedded in $\mathcal{O}$ if $O=\mathcal{O}\cap K$ for some subfield $K\subseteq B_{p,\infty}$.

\subsection{Elliptic curves with level structure and their endomorphism rings}
Arpin \cite{Arpin2203} defined supersingular elliptic curves with level structure and studied their endomorphism rings.
\begin{definition}
Let $\left|\mathcal{S}_{c}\right|$ denote the set of pairs $(E, G)$, up to equivalence $\sim$, where $E$ is a supersingular elliptic curve over $\overline{\mathbb{F}}_{p}$ and $G \subseteq E[c]$ is a cyclic subgroup of order $c$. Two pairs $\left(E_{1}, G_{1}\right),\left(E_{2}, G_{2}\right)$ are equivalent under the equivalence relation $\sim$ if there exists an isomorphism $\rho: E_{1} \rightarrow E_{2}$ such that $\rho\left(G_{1}\right)=G_{2}$. The pairs in $\left|\mathcal{S}_{c} \right|$ are supersingular elliptic curves with level-$c$ structure.
\end{definition}

There exists an isogeny $\phi:E \to E/G$ with $\ker (\phi)=G$, and the dual isogeny $\hat{\phi}:E/G \to E $ has kernel $\hat{G}$.
We define the endomorphism ring of a pair in $\left|\mathcal{S}_{c}\right|$ as following.
\begin{definition}
As a subring of $\operatorname{End}(E)$, we define the ring of endomorphisms of the pair $(E, G) \in\left|\mathcal{S}_{c}\right|$ as follows:
$$
\mathcal{O}(E, G):=\{\alpha \in \operatorname{End}(E): \alpha(G) \subseteq G\}.
$$
\end{definition}

We also write $\mathcal{O}(E, G)$ as $\text{End}(E, G)$. Since $\left|\mathcal{S}_{c}\right|$ is a set of equivalence classes, one can check that $\mathcal{O}(\cdot, \cdot)$ is well-defined on these equivalence classes (see \cite[Proposition 3.3]{Arpin2203}). We consider $\mathcal{O}(\cdot, \cdot)$ as a map that we apply to elements $(E,G)$ of $|\mathcal{S}_{c}|$. Just as supersingular elliptic curves are mapped to the set of maximal orders of $B_{p,\infty}$, we map elements of $|\mathcal{S}_{c}|$ to Eichler orders of level $c$ of $B_{p,\infty}$.

The following proposition is following from Theorem 3.7, Proposition 3.8 and Theorem 4.14 in \cite{Arpin2203}.

\begin{proposition}\label{p21}(\cite[Theorem 3.7, Proposition 3.8 and Theorem 4.14]{Arpin2203})
$\mathcal{O}(E,G)$ is isomorphic to an Eichler order of level $|G|=c$. On the contrary, every Eichler order $\mathcal{O}$ of level $c$ in $B_{p,\infty}$ is isomorphic to some $\mathcal{O}(E,G)$. Moreover, the fiber above $\mathcal{O}(E,G)$ along the map $\mathcal{O}(\cdot, \cdot)$ contains precisely the equivalence classes $(E,G)$, $(E/G,\hat{G})$, $(E^p, G^p)$ and $((E/G)^p,\hat{G}^p)$.
\end{proposition}

We introduce the $(c,\epsilon)$-structure (see \cite{chenu2021higher}) to study the case when $(E/G,\hat{G})$ is isomorphic to $(E^p, G^p)$. If there exists a $c$-isogeny $\phi: E \to E^p$ with $\ker (\phi)=G$, then we have two returning $c$-isogenies:
 $$\phi^p : E^p \to E \quad  \text{and} \quad \hat{\phi}: E^p \to E.$$

\begin{definition}
Let $E$ be a supersingular elliptic curve equipped with a $c$-isogeny $\phi: E \to E^p$ to its conjugate. We say that $(E, \phi)$ is a $(c,\epsilon)$-structure if
$$\hat{\phi}=\epsilon\phi^p \hspace{0.5em} \text{with} \hspace{0.5em} \epsilon \in \{ \pm 1\}.$$
\end{definition}

\begin{remark}
If $c=1$, then $\phi$ is an automorphism and $j(E)=j(E^p) \in \mathbb{F}_p$.
\end{remark}

If $(E, \phi)$ is a $(c,\epsilon)$-structure, then the isogenies $\phi^p$ and $\hat{\phi}$ have the same kernel. It follows that $(E/G,\hat{G})$ is isomorphic to $(E^p, G^p)$. On the contrary, if $(E/G,\hat{G})$ is isomorphic to $(E^p, G^p)$, then $(E,\phi)$ is a $(c,\epsilon)$-structure with $\ker (\phi)=G$ obviously. In this case, the fiber above $\mathcal{O}(E,G)$ is of size $1$ or $2$. In the following, we restrict to the supersingular elliptic curves with $(c, \epsilon)$-structure.

For each $(c,\epsilon)$-structure $(E,\phi)$, we call $\mu:=\pi_p \circ \phi \in \text{End}(E)$ the associated endomorphism of $E$. Notice that $\text{deg}(\mu)=cp$ and $E(\mathbb{F}_{p^2})\cong (\mathbb{Z}/(p+\epsilon)\mathbb{Z})^2$ (see \cite[Proposition 2]{chenu2021higher}). This inspires us to study the $\mathbb{Z}[\sqrt{-cp}]$-oriented elliptic curves.

\subsection{Oriented elliptic curves}
We introduce the definition of orientations, as described by Colò and Kohel in \cite{MR4165916} and Onuki in \cite{MR4170779}.

Let $K$ be an imaginary quadratic field, $O_{K}$ its ring of integers, and $O$ an order in $K$. A $K$-orientation on an elliptic curve $E$ over $\mathbb{F}_{p^{2}}$ is a homomorphism $\iota: K \rightarrow \operatorname{End}^{0}(E)$, where $\operatorname{End}^{0}(E) \cong \operatorname{End}(E) \otimes \mathbb{Q}\cong B_{p,\infty}$. We call the pair $(E, \iota)$ a $K$-oriented elliptic curve. Moreover, we say $\iota$ is an $O$-orientation, and $(E, \iota)$ is an $O$-oriented elliptic curve, if $\iota(O) \subseteq \operatorname{End}(E)$. An $\iota: K \rightarrow \operatorname{End}^{0}(E)$ is primitive if $\iota(O)=\operatorname{End}(E) \cap \iota(K)$, that is, if $\iota$ is "full" in the sense that it does not extend to an $O^{\prime}$-orientation for any strict super-order $O' \supset O$. Orientations are primitive in the rest of this paper.

Given two oriented curves $(E,\iota)$ and $(E',\iota')$, an isogeny $\varphi: E \to E'$ is said to be $K$-oriented if $\iota'=\varphi_{\ast}(\iota)$ with $\varphi_{\ast}(\iota)(\alpha)=\frac{1}{\text{deg}(\varphi)}\varphi \circ \iota(\alpha) \circ \hat{\varphi}$ for any $\alpha \in K$. In this case, we write $\varphi : (E,\iota) \to (E',\iota')$.

Let $O$ be an order in a quadratic field $K$ such that $p$ does not split in $K$ or divide the conductor of $O$. Let $SS_O(p)$ denote the set of $O$-oriented supersingular elliptic curves over $\overline{\mathbb{F}}_{p}$ up to $K$-oriented isomorphism. The subset of primitive $O$-oriented curves is denoted by $SS_O^{pr}(p)$ (see \cite{MR4165916}).

Let $L$ be the ring class field of $O$. There is an extension field $L'$ and prime ideal $\mathfrak{p}$ above $p$ in $L'$ such that every elliptic curve with complex multiplication by $O$ has a representation over $L'$ with good reduction at $\mathfrak{p}$.
Let $\mathcal{E}\ell \ell(O)$ denote the set of isomorphism classes of elliptic curves over $L'$ with endomorphism ring isomorphic to $O$ and good reduction at $\mathfrak{p}$.

By the theory of complex multiplication, we have $\# \mathcal{E}\ell \ell(O)=h(O)$. Moreover, there is a map given by reduction modulo $\mathfrak{p}$:
$$\begin{array}{rll}
    \rho: \mathcal{E}\ell \ell(O) & \longrightarrow & SS_O^{pr}(p) \\
         \tilde{E} & \longmapsto & (E,\iota).
  \end{array}$$
The map $\rho$ is injective on $\mathcal{E}\ell \ell(O)$, so that its image $\rho(\mathcal{E}\ell \ell(O))$ is a subset of $SS_O^{pr}(p)$ of size $h(O)$. We define an action of ideals of $O$ on oriented elliptic curves $(E, \iota) \in SS_O^{pr}(p)$.

Let $(E,\iota) \in SS_O^{pr}(p)$ for some order $O$ of $K$. Let $\mathfrak{a}$ be an integral ideal of $O$ coprime to $p$. Define
$E[\iota(\mathfrak{a})]:=\bigcap_{\alpha \in \mathfrak{a}} \ker (\iota(\alpha)).$
This group defines an isogeny $\phi_{\mathfrak{a}} : E \to E/E[\iota(\mathfrak{a})]$. The action of $\mathfrak{a}$ on $(E,\iota)$ is defined as $\mathfrak{a}\ast (E, \iota):=(\phi_{\mathfrak{a}}(E),(\phi_{\mathfrak{a}})_{\ast} \iota)$.

\begin{proposition}\label{p22}(\cite[Theorem 3.4]{MR4170779} and \cite[Theorem 4.4]{arpin2022orientations})
Assume that $p$ does not split in $K$ and $p$ does not divide the conductor of $O$. Then the ideal action of $C(O)$ defines a free and transitive action on $\rho(\mathcal{E}\ell \ell(O))$. Moreover, $\rho(\mathcal{E}\ell \ell(O))=SS_O^{pr}(p)$ if and only if $p$ ramifies in $K$.
\end{proposition}

From now on, we let $K=\mathbb{Q}(\sqrt{-cp})$, and let $O_K$ be the maximal order of $K$. If $(E,\phi)$ is a supersingular $(c,\epsilon)$-structure and $\mu$ is the associated endomorphism, then
$$\begin{array}{rll}
    \iota_{\varphi}: \mathbb{Q}(\sqrt{-cp}) & \longrightarrow & \operatorname{End}^0(E) \\
         \sqrt{-cp} & \longmapsto & \mu
  \end{array}$$
is a $\mathbb{Z}[\sqrt{-cp}]$-orientation.

\begin{remark}
If $(E,\phi)$ is a supersingular $(c,\epsilon)$-structure, then we have that $E$ is $\mathbb{Z}[\sqrt{-cp}]$-oriented. On the contrary, for every $\mathbb{Z}[\sqrt{-cp}]$-oriented supersingular elliptic curve $E$, there exists an isogeny $\phi: E \to E^p$ with $\deg (\phi)=c$, such that $(E,\phi)$ is a supersingular $(c,\epsilon)$-structure.
\end{remark}

\begin{proposition}(\cite[Proposition 4]{chenu2021higher})
Let $(E,\phi)$ be a supersingular $(c,\epsilon)$-structure with orientation $\iota$.
\begin{enumerate}
  \item If $cp \not\equiv 3 \pmod 4$, then $\iota$ is a primitive $O_K$-orientation.
  \item If $cp \equiv 3 \pmod 4$, then $\iota$ is a primitive $O_K$-orientation if the associated endomorphism $\mu$ fixes $E[2]$ pointwise, and a primitive $\mathbb{Z}[\sqrt{-cp}]$-orientation otherwise.
\end{enumerate}
\end{proposition}
For $cp \equiv 3 \pmod 4$, we call $E$ is in the surface (resp. floor) if $\iota$ is a primitive $O_K$-orientation (resp. $\mathbb{Z}[\sqrt{-cp}]$-orientation).

\subsection{Imaginary quadratic orders and binary quadratic forms}
We recall some basic facts about imaginary quadratic orders and binary quadratic forms. The general references are \cite{MR1012948, MR3236783}.

Let $K$ be an imaginary quadratic field and $O$ an order of $K$. The conductor of $O$ is $f=\left[O_{K}: O\right]$, and the discriminant of $O$ is $f^{2} D_{K}$.

We denote a quadratic form $f(X,Y)=aX^2+bXY+a'Y^2$ by $(a,b,a')$ with $a,b,a' \in \mathbb{Z}$. Its discriminant is $D=b^2-4aa'$. Note that all the forms are positive definite in this paper, which means $D < 0$ and $a>0$. A form $(a,b,a')$ is called primitive if $\text{gcd}(a,b,a')=1$. Two forms $f(X,Y)$ and $g(X,Y)$ are equivalent if there are integers $p$, $q$, $r$ and $s$ such that
$$f(X,Y)=g(pX+qY,rX+sY) \quad \text{and} \quad ps-qr=1. $$
In this case, we denote $f\sim g$.
 If a form $(a,b,a')$ satisfies
\begin{itemize}
  \item [(1)] $-a < b \le a$,
  \item [(2)] $ a\le a' $,
  \item [(3)] if $a=a'$ then $b \ge 0$,
\end{itemize}
then it is reduced. Every primitive positive definite form is equivalent to a unique reduced form (see \cite[Theorem 2.8]{MR3236783}).

For every $D=b^2-4aa'<0$, there is an imaginary quadratic order $O$ such that the discriminant of $O$ is $D$.
By \cite[Theorem 6.15]{MR1012948}, every ideal in $O$ has an integral basis, and we use $[a,\beta]=\mathbb{Z}a+\mathbb{Z}\beta$ to represent an ideal in $O$, where $a$ is a rational integer and $\beta$ is a quadratic algebra integer.

We can relate the ideal class group $C(O)$ and the form class group $C(D)$ defined in \cite[\S 3]{MR3236783} as follows:
\begin{proposition}\label{t3}(\cite[Theorem 7.7]{MR3236783})
 Let $O$ be the order of discriminant $D$ in an imaginary quadratic field $K$. Then:
\begin{itemize}
  \item [(1)] If $f(X,Y)=aX^2+bXY+a'Y^2$ is a primitive positive definite quadratic form of discriminant $D$, then $[a, (-b+\sqrt{D})/2]$ is a proper ideal of $O$.
  \item [(2)] The map sending $f(X,Y)$ to $[a, (-b+\sqrt{D})/2]$ induces an isomorphism $\iota : C(D) \longrightarrow C(O)$.
  \item [(3)] A positive integer $m$ is represented by a form $f(X,Y)$ if and only if $m$ is the norm of some ideal $\mathfrak{a}$ in the corresponding ideal class in $C(O)$.
\end{itemize}
\end{proposition}

\begin{remark}
If $\mathfrak{a}=[a,\beta]$ is a proper $O$-ideal with $\operatorname{Im}(\beta / a)>0$, then
$$f(X,Y)=\frac{N(a X-\beta Y)}{N(\mathfrak{a})}$$
is a positive definite form of discriminant $D$. On the level of classes, this map is the inverse to the map $\iota$ in Proposition \ref{t3}.
\end{remark}

By Proposition \ref{t3}, one can use the primitive reduced quadratic forms to represent the ideal classes in an imaginary quadratic order, and the composition of forms is equivalent to the multiplication of ideals.

\subsection{Genus theory}

Let $O$ be an imaginary quadratic order with discriminant $D< 0$. Let $p_1, \ldots, p_r$ be the distinct odd primes dividing $D$. Then let us consider the following characters:
\begin{itemize}
  \item [(1)] $\chi_i(a)=\left( \frac{a}{p_i}\right)$ \quad \quad \quad \ defined for $a$ prime to $p_i$, $i=1,\ldots,r$
  \item [(2)] $\delta(a)=(-1)^{(a-1)/2}$ \ \ \ defined for $a$ odd
  \item [(3)] $\varepsilon(a)=(-1)^{(a^2-1)/8}$ \ \ defined for $a$ odd
\end{itemize}

When $D \equiv 1 \pmod 4$, we define $\chi_1,\ldots, \chi_r$ to be the assigned characters. When $D \equiv 0 \pmod 4$, we write $D=-4n$, and the assign characters are defined by the following table:
\begin{table}[H]
    \centering
  \begin{tabular}{c|l}
    \hline
    $n$ & assigned characters \\
    \hline
    $n \equiv \ 3 \pmod  4$ & $\chi_1,\ldots,\chi_r$ \\
    $n \equiv \ 1 \pmod  4$ & $\chi_1,\ldots,\chi_r,\delta$ \\
    $n \equiv \ 2 \pmod  8$ & $\chi_1,\ldots,\chi_r,\delta\varepsilon$ \\
    $n \equiv \ 4 \pmod 8$ & $\chi_1,\ldots,\chi_r,\delta$ \\
    $n \equiv \ 6 \pmod  8$ & $\chi_1,\ldots,\chi_r,\varepsilon$\\
    $n \equiv \ 0 \pmod  8$ & $\chi_1,\ldots,\chi_r,\varepsilon,\delta$\\
    \hline
  \end{tabular}
  \end{table}

\begin{proposition}\label{p1}(\cite[Proposition 3.11 and Theorem 3.15]{MR3236783})
  Let $D\equiv 0,1 \pmod  4$ be negative. If the number of assigned characters of $D$ is $\mu$, then there are $2^{\mu-1}$ genera of forms of discriminant $D$. Moreover, the class group $C(D)$ has exactly $2^{\mu-1}$ elements of order $\le 2$.
\end{proposition}

Let $p>2$ and $c$ be different primes. If $cp \equiv 3 \pmod 4$, then the assigned characters of $D=-cp$ are $\chi_1(a)=\left( \frac{a}{p} \right)$ and $\chi_2(a)=\left( \frac{a}{c} \right)$. Moreover, the assigned characters of $D=-16cp$ are $\chi_1$, $\chi_2$ and $\delta$. Note that $\chi_2(a)=\left( \frac{a}{2} \right)=\varepsilon(a)$ if $c=2$.

\section{Constructing Eichler orders}
Let $p$ and $c$ be different prime integers. Choose a prime integer $q$ such that

\begin{equation}\label{e1}
\left \{
\begin{array}{ll} \vspace{1ex}
  q \equiv 3 \pmod 8, \quad \left( \frac{p}{q} \right) =-1, \quad  \left( \frac{c}{q} \right) = 1, & \text{if}  \ c \ \text{is odd}; \\
  q \equiv 7 \pmod 8, \quad \left( \frac{p}{q} \right) =-1, & \text{if} \ c=2.
\end{array} \right.
\end{equation}

Let $B_{p,\infty}$ denote the quaternion algebra over $\mathbb{Q}$ ramified precisely at $p$ and $\infty$. By \cite[Lemma 1.1]{MR683249}, the following proposition holds.

\begin{proposition}\label{p31}
Let $q$ be a prime satisfying $(\ref{e1})$.
There exist $\alpha'$ and $\beta$ satisfying $\alpha'^2=-p$, $\beta^2=-q$ and $\alpha' \beta=-\beta \alpha'$, such that $B_{p,\infty } \cong \mathbb{Q}+\mathbb{Q}\alpha' +\mathbb{Q} \beta +\mathbb{Q}\alpha' \beta$.
\end{proposition}

Moreover, we have the following proposition.
\begin{proposition}\label{p32}
Notations being as Proposition \ref{p31}.
There exists an $\alpha \in B_{p,\infty}$ such that $\alpha^2=-cp$ and $\alpha \beta =-\beta  \alpha$. Moreover, we have $B_{p,\infty } \cong \mathbb{Q}+\mathbb{Q}\alpha +\mathbb{Q} \beta +\mathbb{Q}\alpha \beta$.
\end{proposition}

\begin{proof}
For an element $\gamma =x \alpha' +y \alpha' \beta \in B_{p,\infty } $, we have $\text{Trd}(\gamma)=0$, $\text{Nrd}(\gamma)=x^2p+y^2pq$ and $\gamma \beta =-\beta \gamma$.

By Hasse-Minkowski Theorem (see \cite[Theorem 8, Chapter 4]{MR0344216}), there exist $i,j \in \mathbb{Q}$ with $j \neq 0$ such that $i^2p+j^2pq=cp$. Let $\alpha = i \alpha' +j \alpha' \beta$, and we have that $\alpha^2=-cp$ and $\alpha \beta =-\beta  \alpha$. Moreover, we have $\alpha' = \frac{i \alpha -j \alpha \beta}{c}$. It follows $B_{p,\infty } \cong \mathbb{Q}+\mathbb{Q}\alpha +\mathbb{Q} \beta +\mathbb{Q}\alpha \beta$.

\end{proof}

\begin{remark}
As we see, $(-p,-q)_{\ell}=(-cp,-q)_{\ell}=-1$ if and only if $\ell=p$ and $\infty$. So $B_{p,\infty} \cong H(-p,-q) \cong H(-cp,-q)$ for any prime $q$ satisfying $(\ref{e1})$. This is another way to prove Proposition \ref{p32}.

\end{remark}

Let $q$ be a prime satisfying $(\ref{e1})$. Choose an integer $r$ such that
\begin{equation}\label{E2}
r^2+cp \equiv 0 \pmod q.
\end{equation}

Besides, when $cp \equiv 3 \pmod 4$, we choose an integer $r'$ such that
\begin{equation}\label{E3}
r'^2+cp \equiv 0 \pmod {4q}.
\end{equation}

\begin{theorem}
Denote notions as $(\ref{e1})$, $(\ref{E2})$ and $(\ref{E3})$, the lattice $\mathcal{O}_c(q,r) = \mathbb{Z}+\mathbb{Z}\frac{1+\beta}{2}+\mathbb{Z}\frac{\alpha(1+\beta)}{2} +\mathbb{Z}\frac{(r+\alpha)\beta}{q}$ is an Eichler order in $B_{p, \infty}$ with level $c$. Moreover, if $cp \equiv 3 \pmod 4$, the lattice $\mathcal{O}'_c(q,r') = \mathbb{Z}+\mathbb{Z}\frac{1+\alpha}{2}+\mathbb{Z}{\beta} +\mathbb{Z}\frac{(r'+\alpha)\beta}{2q}$ is also an Eichler order in $B_{p, \infty}$ with level $c$.
\end{theorem}

\begin{proof}
  It is easy to see that $\mathcal{O}_c(q,r)$ and $\mathcal{O}'_c(q,r')$ contain $\mathbb{Z}+\mathbb{Z}\alpha +\mathbb{Z} \beta + \mathbb{Z} \alpha \beta$, so $\mathcal{O}_c(q,r)$ and $\mathcal{O}'_c(q,r')$ contain a basis of $B_{p, \infty}$. By routine calculations, we have that $\mathcal{O}_c(q,r)$ and $\mathcal{O}'_c(q,r')$ are subrings of $B_{p, \infty}$ and their discriminants are equal to $c^2p^2$, which means they are Eichler orders in $B_{p, \infty}$ with level $c$.

\end{proof}

\begin{remark}
  Note that the isomorphism classes of $\mathcal{O}_c(q,r)$ or $\mathcal{O}'_c(q,r')$ depend on $q$, but do not depend on $r$ or $r'$.
\end{remark}

Let $I(2)$ be the group of fractional ideals of $K=\mathbb{Q}(\sqrt{-cp})$ which are prime to $2$. Let $P(2)$ be the group of principal ideals in $I(2)$. Let $P_{\mathbb{Z}}(2)$ be the
subgroup of $I(2)$ generated by principal ideals of the form $\alpha O_K$, where $\alpha \in O_K$ satisfies $\alpha \equiv 1 \pmod {2O_K}$.
For any $q$ satisfying $(\ref{e1})$, we know $q$ splits in $K$ because $\left( \frac{-cp}{q} \right)=1$.

\begin{proposition}\label{p33}
Suppose $q_{1} \neq q_{2}$ are primes satisfying $(\ref{e1})$. Let $K=\mathbb{Q}(\sqrt{-cp})$. Suppose $q_{1}$ and $q_{2}$ have prime decompositions $q_{1} O_{K}=\mathfrak{q}_{1} \overline{\mathfrak{q}}_{1}$ and $q_{2} O_{K}=\mathfrak{q}_{2} \overline{\mathfrak{q}}_{2}$.
\begin{itemize}
  \item [(1)] $\mathcal{O}'_c(q_{1},r'_1) \cong \mathcal{O}'_c(q_{2},r'_2) $ if and only if either $\mathfrak{q}_{1} \mathfrak{q}_{2} \in P(2)$ or $\mathfrak{q}_{1} \overline{\mathfrak{q}}_{2} \in P(2)$.
  \item [(2)] $\mathcal{O}_c(q_{1},r_1) \cong \mathcal{O}_c(q_{2},r_2)$ if and only if either $\mathfrak{q}_{1} \mathfrak{q}_{2} \in$ $P_{\mathbb{Z}}(2)$ or $\mathfrak{q}_{1} \overline{\mathfrak{q}}_{2} \in P_{\mathbb{Z}}(2)$.
\end{itemize}
\end{proposition}

\begin{proof}
The proof is similar to that of Proposition 2.1 in \cite{MR683249}, and we just provide a sketch of the proof.

We just discuss the isomorphism classes of $\mathcal{O}_c(q,r)$, and similar results hold for $\mathcal{O}'_c(q,r')$. Put $\mathcal{O}_c(q_1,r_1) = \mathbb{Z}+\mathbb{Z}\frac{1+\beta_1}{2}+\mathbb{Z}\frac{\alpha(1+\beta_1)}{2} +\mathbb{Z}\frac{(r_1+\alpha)\beta_1}{q_1}$ and $\mathcal{O}_c(q_2,r_2) = \mathbb{Z}+\mathbb{Z}\frac{1+\beta_2}{2}+\mathbb{Z}\frac{\alpha(1+\beta_2)}{2} +\mathbb{Z}\frac{(r_2+\alpha)\beta_2}{q_2}$.

Let $\theta$ be an isomorphism of $\mathcal{O}_c(q_{1},r_1)$ to $\mathcal{O}_c(q_{2},r_2)$. Now put $\theta((1+\beta_1)/2)=w+x(1+\beta_2)/2+y\alpha(1+\beta_2)/2 +z(r_2+\alpha)\beta_2/q_2$. Since $\theta(\alpha)\theta(\beta_1)=-\theta(\beta_1)\theta(\alpha)$, we have $y=0$. As $\text{Trd}((1+\beta_1)/2)=1$ and $\text{Nrd}((1+\beta_1)/2)=(1+q_1)/4$, we have $(xq_1+2zr_1)^2+4z^2cp=q_1q_2$. It follows that either $\mathfrak{q}_{1} \mathfrak{q}_{2} \in$ $P_{\mathbb{Z}}(2)$ or $\mathfrak{q}_{1} \overline{\mathfrak{q}}_{2} \in P_{\mathbb{Z}}(2)$.

On the contrary, assume that $x^2+4y^2cp=q_1q_2$ for some $x,y \in \mathbb{Z}$. We have $x^2+4y^2cp\equiv x^2-4r_2^2y^2 \equiv (x-2r_2y)(x+2r_2y) \pmod {q_2}$. Changing the sign of $y$ if necessary, we get $x \equiv 2r_2y \pmod {q_2}$. Denote by $\vartheta$ the $\mathbb{Q}$ linear map of $B_{p,\infty}$ to $B_{p,\infty}$ such that
$$\begin{array}{rcl}
    \vartheta(1) & = & 1, \\
    \vartheta(\alpha) & = & \varepsilon \alpha, \\
    \vartheta(\beta_1) & = & (x-2r_2y) \beta_2/q_2 +2y(r_2+\alpha)\beta_2/q_2 , \\
    \vartheta(\alpha\beta_1) & =  &\vartheta(\alpha)\vartheta(\beta_1),
  \end{array}
$$
where $\varepsilon=1$ or $-1$. We can show that $\vartheta$ induces an isomorphism of $\mathcal{O}_c(q_1,r_1)$ and $\mathcal{O}_c(q_2,r_2)$, and we omit it here.
\end{proof}

\begin{remark}
Denote by $e(q)$ (resp. $e'(q)$) the half of the cardinality of the group of the units of $\mathcal{O}_c(q,r)$ (resp. $\mathcal{O}'_c(q,r')$). As Lemma 1.8 in \cite{MR683249}, $\mathcal{O}_c(q_{1},r_1) \cong \mathcal{O}'_c(q_{2},r'_2)$ if and only if $e(q_1)=e'(q_2)=2$. In this case, we can show that $\mathcal{O}_c(q_{1},r_1)$ (resp. $\mathcal{O}'_c(q_{2},r'_2)$) corresponds to $(E(1728),G)$ (resp. $(E(1728),G')$) respectively.
\end{remark}

For $cp \equiv 3 \pmod 4$, $\mathfrak{q}_1 \mathfrak{q}_2 \in P(2)$ (resp. $\mathfrak{q}_{1} \overline{\mathfrak{q}}_{2} \in P(2)$) if and only if the ideals $\mathfrak{q}_1$ and $\bar{\mathfrak{q}}_2$ (resp. $\mathfrak{q}_2$) are equivalent in the class group $C(-cp)$ which means $q_1$ and $q_2$ can be represented by the same reduced quadratic forms with discriminant $-cp$. Note that $q_i$ can be represented by  a form $f$ if and only if it can be represented by $f^{-1}$. By \cite[Proposition 4.3]{xiao2022endomorphism}, we have the following corollary.

\begin{corollary} \label{c1}
Notations being as Proposition \ref{p33}.

(1) $\mathcal{O}'_c(q_{1},r'_1) \cong \mathcal{O}'_c(q_{2},r'_2)$ if and only if $q_1$ and $q_2$ can be represented by the same reduced forms with discriminant $-cp$.

(2) $\mathcal{O}_c(q_{1},r_1) \cong \mathcal{O}_c(q_{2},r_2)$ if and only if $q_1$ and $q_2$ can be represented by the same reduced forms with discriminant $-16cp$ .
\end{corollary}

Let us discuss the number of isomorphism classes of $\mathcal{O}_c(q,r)$ and $\mathcal{O}'_c(q,r')$. Notations being as subsection 2.6. By genus theory, the reduced quadratic forms with discriminant $D$ can be divided into $2^{\mu-1}$ many classes, and each class contains $h(D)/2^{\mu-1}$ many reduced forms.

Let $\Lambda(q)=(\chi_1(q),\chi_2(q),\delta(q))$ (resp. $\Lambda'(q)=(\chi_1(q),\chi_2(q))$) be the assigned values of $q$ associated with $-16cp$ (resp. $-cp$), where $\chi_1(q)=\left( \frac{q}{p} \right)$, $\chi_2(q)=\left( \frac{q}{c} \right)$ and $\delta(q)=(-1)^{(q-1)/2}=-1$. The forms with associated values $\Lambda(q)$ (resp. $\Lambda'(q)$) are called in the genus class $\Lambda(q)$ (resp. $\Lambda'(q)$). Note that the form $f$ is also in the genus class $\Lambda(q)$ (resp. $\Lambda'(q)$) if and only if $f^{-1}$ is in the genus class $\Lambda(q)$ (resp. $\Lambda'(q)$). Moreover, $f=f^{-1}$ if and only if the order of $f$ is equal to $1$ or $2$. By Corollary \ref{c1}, we have the following theorem.

\begin{theorem}\label{zt2}
If $\eta$ (resp. $\eta'$) is the number of reduced forms with order $\le 2$ in the genus class $\Lambda(q)$ (resp. $\Lambda'(q)$), then the number of isomorphism classes of $\mathcal{O}_c(q,r)$ (resp. $\mathcal{O}'_c(q,r)$) is $h(-4cp)/4+\eta/2$ (resp. $h(-cp)/4+\eta'/2$).
\end{theorem}

Assume that $c$ is an odd prime. There are two forms $(1,1,\frac{1+cp}{4})$ and $(c,c,\frac{c+p}{4})$ with order $\le 2$ in the class group $C(-cp)$. There are four forms $(1,0,4cp)$, $(4,0,cp)$, $(c,0,4p)$ and $(4c,0,p)$ with order $\le 2$ in the class group $C(-16cp)$. If $c=2$, then there are four reduced forms $(1,0,8p)$, $(8,0,p)$, $(4,4,1+2p)$ and $(8,8,2+p)$ with order $\le 2$ in the class group $C(-32p)$.

By simple calculations, we give the number of $\eta$ and $\eta'$ and reduced forms with order $\le 2$ in the genus class $\Lambda(q)$ and $\Lambda'(q)$) in the following two tables.

 \begin{table}[H]
   \centering
    \caption{Reduced forms with order $\le 2$ in the genus class $\Lambda'(q)$}
   \begin{tabular}{cccccl}

	\toprule
	$p \pmod 4$  &$c \pmod 4$ &$\Lambda'(q)$    &$\left( \frac{c}{p} \right )$ &$\eta'$   &forms \\
	\midrule
	\multirow{2}{*}{$3$} &\multirow{2}{*}{$1$} &\multirow{2}{*}{$(+1,+1)$} &$+1$ &$2$ &$(1,1,\frac{1+cp}{4}),(c,c,\frac{c+p}{4})$\\
   &&&$-1$ &$1$ &$(1,1,\frac{1+cp}{4})$\\
   \midrule
	\multirow{2}{*}{$1$} &\multirow{2}{*}{$3$} &\multirow{2}{*}{$(-1,-1)$} &$-1$ &$1$ &$(c,c,\frac{c+p}{4})$\\
   &&&$+1$ &$0$ &\\
    \bottomrule
\end{tabular}
 \end{table}

 \begin{table}[H]
   \centering
    \caption{Reduced forms with order $\le 2$ in the genus class $\Lambda(q)$}
   \begin{tabular}{cccccl}

	\toprule
	$p \pmod 4$  &$c \pmod 4$ &$\Lambda(q)$    &$\left( \frac{c}{p} \right )$ &$\eta$   &forms \\
	\midrule
	\multirow{2}{*}{$3$} &\multirow{2}{*}{$1$} &\multirow{2}{*}{$(+1,+1,-1)$} &$+1$ &$2$ &$(4,0,cp),(4c,0,p)$\\
   &&&$-1$ &$1$ &$(4,0,cp)$\\
   \midrule
	\multirow{2}{*}{$1$} &\multirow{2}{*}{$3$} &\multirow{2}{*}{$(-1,-1,-1)$} &$-1$ &$1$ &$(c,0,4p)$\\
   &&&$+1$ &$0$ &\\
   \midrule
	\multirow{2}{*}{$3$} &\multirow{2}{*}{$3$} &\multirow{2}{*}{$(+1,-1,-1)$} &$+1$ &$2$ &$(c,0,4p),(4c,0,p)$\\
   &&&$-1$ &$0$ &\\
   \midrule
	$1$ &$1$ &$(-1,+1,-1)$ & &$0$ &\\
  \midrule
	\multirow{2}{*}{$3$} &\multirow{2}{*}{$2$} &\multirow{2}{*}{$(+1,+1,-1)$} &$+1$ &$2$ &$(4,4,1+2p),(8,0,p)$\\
   &&&$-1$ &$1$ &$(4,4,1+2p)$\\
   \midrule
	\multirow{2}{*}{$1$} &\multirow{2}{*}{$2$} &\multirow{2}{*}{$(-1,+1,-1)$} &$-1$ &$1$ &$(8,8,2+p)$\\
   &&&$+1$ &$0$ &\\
    \bottomrule
\end{tabular}
 \end{table}

\section{Eichler orders and supersingular j-invariants}

By Proposition \ref{p21}, the fiber above the Eichler order $\mathcal{O}_c(q,r)$ or $\mathcal{O}'_c(q,r')$ along the map $\mathcal{O}(\cdot,\cdot)$ contains the equivalence classes $(E,G)$, $(E/G,\hat{G})$, $(E^p,G^p)$ and $((E/G)^p,\hat{G}^p)$ with $|G|=|\hat{G}|=c$. Moreover, we have $\mathbb{Z}[\sqrt{-cp}] \subset \mathcal{O}_c(q,r)$ and $\mathbb{Z}[\sqrt{-cp}] \subset \mathcal{O}'_c(q,r')$, so $(E/G,\hat{G})$ is isomorphic to $(E^p,G^p)$ and $(E,G)$ is isomorphic to $((E/G)^p,\hat{G}^p)$.

In the following, we will show that the Eichler order $\mathcal{O}_c(q,r)$ or $\mathcal{O}'_c(q,r')$ can represent a binary quadratic form in the genus class $\Lambda(q)$ or $\Lambda'(q)$ respectively. On the contrary, every reduced binary quadratic form in the genus class $\Lambda(q)$ or $\Lambda'(q)$ corresponds to a pair of supersingular $j$-invariants $(j,j^p)$. As in \cite{MR3240797}, we define $\mathcal{O}_c(q,r)^{T}=\{ 2\gamma-\text{Trd}(\gamma) : \gamma \in \mathcal{O}_c(q,r)\}$ and $\mathcal{O}'_c(q,r')^{T}=\{ 2\gamma-\text{Trd}(\gamma) : \gamma \in \mathcal{O}'_c(q,r')\}$.

\begin{proposition}
If $\gamma \in \mathcal{O}_c(q,r)$ (resp. $\mathcal{O}'_c(q,r')$) with $\operatorname{Trd}(\gamma \alpha)=0$, then the reduced norm $\operatorname{Nrd}(2\gamma-\operatorname{Trd}(\gamma))$ (resp. $\operatorname{Nrd}(\gamma-\operatorname{Trd}(\gamma)/2)$) is a binary quadratic form in the genus class $\Lambda(q)$ (resp. $\Lambda'(q)$) respectively. Denote its reduced form by $(a,b,a')$. If $a>4$ (resp. $a>1$), then the first two successive minima of $\mathcal{O}_c(q,r)^T$ (resp. $\mathcal{O}'_c(q,r')^T$) are $a$ (resp. $4a$) and $a'$ (resp. $4a'$) respectively.
\end{proposition}
\begin{proof}
We can write $\gamma=w+x \frac{1+\beta}{2}+y \frac{\alpha(1+\beta)}{2}+z \frac{(r+\alpha)\beta}{q} \in \mathcal{O}_c(q,r)$ with $w,x,y,z \in \mathbb{Z}$. We have
$$\operatorname{Trd}(\gamma)= 2w + x $$ and $$\operatorname{Nrd}(\gamma)= \left( w + \frac{x}{2} \right)^2 + \left(\frac{x}{2}+\frac{rz}{q} \right)^2q + \frac{y^2}{4}cp+\left(\frac{y}{2}+\frac{z}{q} \right)^2cpq.$$
Moreover, we have
$$\operatorname{Nrd}(2\gamma-\operatorname{Trd}(\gamma))= cp y^2 + q \left( x+\frac{2zr}{q} \right)^2 + {cpq}\left( y+\frac{2z}{q}\right)^2,$$
which is a ternary quadratic form.
As we see, $\operatorname{Trd}(\gamma \alpha)=0$ if and only if $y=0$, and we have
$$\operatorname{Nrd}(2\gamma-\operatorname{Trd}(\gamma))=q x^2+4rxz+\frac{4r^2+4cp}{q}z^2,$$
which is a binary quadratic form in the genus class $\Lambda(q)$.

Assume that the form $(q,4r, \frac{4r^2+4cp}{q})$ is equivalent to a reduced form $(a,b,a')$ where $a<4\sqrt{cp/3}$ and $a \le a' < cp$ if $a>4$. We have $\text{Nrd}(2\gamma-\text{Trd}(\gamma))>cp$ if $y \neq 0$, so the first two successive minima of $\mathcal{O}_c(q,r)^T$ are $a$ and $a'$ if $a>4$.

Similarly, if $\gamma \in \mathcal{O}'_c(q,r')$ with $\operatorname{Trd}(\gamma \alpha)=0$, then the reduced norm equation $\operatorname{Nrd}(\gamma-\operatorname{Trd}(\gamma)/2)$ is a binary quadratic form $(q,r', \frac{(r')^2+cp}{4q})$ in the genus class $\Lambda'(q)$. Assume that the form $(q,r', \frac{(r')^2+cp}{4q})$ is equivalent to a reduced form $(a,b,a')$ where $a<\sqrt{cp/3}$ and $a < a' < cp/4$ if $a>1$. It follows that the first two successive minima of $\mathcal{O}'_c(q,r')^T$ are $4a$ and $4a'$.

\end{proof}
\begin{remark}
If $\gamma \in \mathcal{O}(q,r)$ (resp. $O'(q,r')$), then the reduced norm $\operatorname{Nrd}(2 \gamma-\operatorname{Trd}(\gamma))$ (resp. $\operatorname{Nrd}( \gamma-\operatorname{Trd}(\gamma)/2)$) is a ternary quadratic form. Moreover, this ternary quadratic form can represent a binary quadratic form $(q,4r,\frac{4r^2+4cp}{q})$ (resp. $(q,r,\frac{(r')^2+cp}{4q})$) in the genus class $\Lambda(q)$ (resp. $\Lambda'(q)$) (see \cite[Chapter 3]{Dickson1930}). In this paper, we say that each Eichler order $\mathcal{O}_c(q,r)$ (resp. $\mathcal{O}'_c(q,r')$) can represent a binary quadratic form in the genus class $\Lambda(q)$ (resp. $\Lambda'(q)$).
\end{remark}

\begin{remark}
If the Eichler order $\mathcal{O}_c(q,r)$ can represent $(3,*,*)$ (resp. $(4,*,*)$), then the imaginary quadratic order $\mathbb{Z}[\frac{1+\sqrt{-3}}{2}]$ (resp. $\mathbb{Z}[\sqrt{-1}]$)) can be embedded into $\mathcal{O}_c(q,r)$. It follows that $\mathcal{O}_c(q,r)$ corresponds to the supersingular $j$-invariant $0$ (resp. $1728$). If the Eichler order $\mathcal{O}'_c(q,r')$ can represent $(1,*,*)$, then it corresponds to the supersingular $j$-invariant $1728$.
\end{remark}

On the contrary, we want to show that every form in the genus class $\Lambda(q)$ or $\Lambda'(q)$ corresponds to a pair of equivalence classes $(E,G)$ and $(E^p,G^p)$.

Let $D$ be the discriminant of an imaginary quadratic order $O$. The Hilbert class polynomial \cite[\S 13]{MR3236783} $H_D(X)$ is defined as
$$H_D(X):=\prod_{j(E) \in \text{Ell}_O(\mathbb{C})}(X-j(E)),$$
where $\text{Ell}_O(\mathbb{C})=\{j(E) \mid \text{End}(E/\mathbb{C}) \cong O\}$.
We study the number of common roots of different Hilbert class polynomials. Write
$$J(D_1,D_2)=\prod_{\substack{[\tau_1],[\tau_2] \\ \text{disc}(\tau_i)=D_i}}(j(\tau_1)-j(\tau_2)),$$
where $[\tau_i]$ runs over all elements of the upper half-plane with discriminant $D_i$ modulo $\text{SL}_2(\mathbb{Z})$. $J(D_1,D_2)$ can be viewed as the resultant of Hilbert class polynomials $H_{D_1}(X)$ and $H_{D_2}(X)$.

Gross and Zagier \cite{MR772491} studied the prime factorizations of $J(D_1,D_2)$ where $D_1$ and $D_2$ are two fundamental discriminants which are relatively prime. Recently, Lauter and Viray \cite{MR3431591} studied the prime factorizations of $J(D_1,D_2)$ for arbitrary $D_1$ and $D_2$. Let $w_i$ denote the number of roots of unity in the imaginary quadratic order of discriminant $D_i$. Let $f_i$ denote the conductor of $D_i$. The following two propositions are from \cite{MR3431591}.

\begin{proposition}\label{t5}(\cite[Theorem 1.1]{MR3431591})
Let $D_1$, $D_2$ be any two distinct discriminants. Then there exists a function $F$ that takes non-negative integers of the form $(D_1 D_2 - x^2)/4$ to (possibly fractional) prime powers. This function satisfies
$$J(D_1, D_2)^{\frac{8}{w_1 w_2}} = \pm \prod_{\substack{x^2 \le D_1D_2 \\ x^2 \equiv D_1 D_2 \pmod 4}} F \left (\frac{D_1 D_2 - x^2}{4} \right ).$$
Moreover, $F(m) = 1$ unless either (1) $m = 0$ and $D_2 = D_1 \ell^{2k}$ for some prime $\ell$ or (2) the Hilbert symbol $(D_1, -m)_{\ell} = -1$ at a unique finite prime $\ell$ and this prime divides $m$. In both of these cases, $F(m)$ is a (possibly fractional) power of $\ell$.
\end{proposition}

\begin{proposition}\label{t6}(\cite[Theorem 1.5]{MR3431591})
Let $m$ be a positive integer of the form $(D_1D_2-x^2)/{4}$ and $\ell>2$ a fixed prime. If $m\ell$ is coprime to $f_1$, then we have
$$v_{\ell}(F(m))= \left \{
\begin{array}{ll}
  \frac{1}{e}\rho(m)\sum_{r \geq 1} \mathfrak{A}(m/\ell^r) & \text{if} \ \ell \nmid f_2 , \\
  \rho(m)\mathfrak{A}(m/\ell^{1+v(f_2)}) & \text{if} \ \ell \mid f_2 ,
\end{array}
\right. $$
where $e$ is the ramification degree of $\ell$ in $\mathbb{Q}(\sqrt{D_1})$ and
$$\rho(m)=\left \{
\begin{array}{ll}
  0 & \text{if} \ (D_1,-m)_p=-1 \ \text{for} \ p\mid D_1, p\nmid f_1 \ell, \\
  2^{\# \{p|(m,D_1):p\nmid f_2 \ \text{or} \ p=\ell \} } & \text{otherwise},
\end{array}
\right. $$
$$\mathfrak{A}(N)=\# \left \{
\begin{array}{cl}
   & N(\mathfrak{b})=N, \mathfrak{b} \ \text{invertible,} \\
  \mathfrak{b} \subseteq O_{D_1} & p\nmid \mathfrak{b} \ \text{for all} \ p\mid(N,f_2),p\nmid \ell D_1 \\
   & \mathfrak{p}^3 \nmid \mathfrak{b} \ \text{for all} \ \mathfrak{p}\mid p \mid (N,f_2,D_1), p\neq \ell
\end{array}
\right \}.$$
\end{proposition}

In the rest of this paper, we denote $D_1=-a$ (resp. $D_1=-4a$) and $D_2=-a'$ (resp. $D_2=-4a'$) if the reduced form $(a,b,a')$ belongs to the genus class $\Lambda (q)$ (resp. $\Lambda '(q)$). Denote $\overline{H}_{D_i}(X)=H_{D_i}(X)$ mod $p$. We show that $\gcd(\overline{H}_{D_1}(X),\overline{H}_{D_2}(X))=(X-j)$ with $j \in \mathbb{F}_p$ if the form $(a,b,a')$ has order $\le 2$ in \ref{Appendix}.

If the order of $(a,b,a')$ is bigger than $2$ and $|D_1| \le 4$, then we must have $a=3$. Since the form $(a,b,a')$ is in the class $\Lambda(q)$, we have $p \equiv 2 \pmod 3$, $c \equiv 1 \pmod 3$, $b=\pm 2$ and $a'=\frac{1+4cp}{3}$. As we know $0$ is the only root of the Hilbert class polynomial $H_{-3}(X)=X$. By Proposition \ref{t6}, we have $p \mid J(-3,-\frac{1+4cp}{3})$. It follows $\gcd(\overline{H}_{-3}(X),\overline{H}_{-\frac{1+4cp}{3}}(X))=X$. For other forms in the genus class $\Lambda (q)$ and $\Lambda'(q)$, we have the following proposition.

\begin{proposition}\label{p5}
Assume that $c < 3p/16$ is a prime. Let $(a,b,a')$ be a primitive reduced form in the genus class $\Lambda (q)$ (resp. $\Lambda'(q)$) with $a > 4$ (resp. $a > 1$). If the order of $(a,b,a')$ is bigger than $2$, then we have $p^2 \parallel J(-a,-a')$ (resp. $p^2 \parallel J(-4a,-4a')$).
\end{proposition}

\begin{proof}
If the form $(a,b,a')$ belongs to the genus class $\Lambda (q)$, then $b^2-4aa'=-16cp$. Assume that there exist integers $x$ and $k\ge 0$ such that $aa'-x^2=4kp$. We have the following inequation:
$$x^2 \le aa'=4cp+b^2/4 < 4cp + 4cp/3 =16cp/3.$$
This inequality holds since $|b|<a< \sqrt{16cp/3}$. Assume $c < 3p/16$, we have $x^2 < p^2$. Since $x^2 \equiv aa' \equiv b^2/4 \pmod p$ and $|x| < p$, we have $x = \pm b/2$ or $\pm (|b|/2-p)$. If $x=\pm (|b|/2-p)$, then $aa'-x^2=aa'-(|b|/2-p)^2=4cp+|b|p-p^2<0$ since $|b|<c< 3p/16$. It follows that $x = \pm b/2$ and $k=c$.

Let $m=cp$ and $\ell =p$. By Proposition \ref{t6}, we have $e=1$, $\rho(cp)=1$ for any $(a,b,a')$ with order $>2$. Moreover, $\mathfrak{A}(c)=2$ since $\left( \frac{-a}{c} \right)=\left( \frac{-q}{c} \right)=1$, and $\mathfrak{A}(cp/p^r) = 0$ for any $r>1$. We also have $w_1=w_2=2$ since $|D_2|> |D_1|>4$. It follows that $p^2 \parallel J(D_1,D_2)$.

The proof is similar if the form $(a,b,a')$ belongs to the genus class $\Lambda'(q)$.

\end{proof}

As defined before this proposition, we have that there exists at least one common root of $\overline{H}_{D_1}(X)$ and $\overline{H}_{D_2}(X)$.

\begin{theorem}\label{zt3}
Assume that $c < 3p/16$ is a prime. Let $(a,b,a')$ be a reduced form in the genus class $\Lambda (q)$ or $\Lambda' (q)$. If $j$ is a common root of $\overline{H}_{D_1}(X)$ and $\overline{H}_{D_2}(X)$, then $j\in \mathbb{F}_p$ if and only if there exist $x, y \in \mathbb{Z}$ with $\gcd(x,y)=1$ such that $x^2+ay^2=4c$ or $x^2+ay^2=c$ respectively. In this case, $\gcd(\overline{H}_{D_1}(X),\overline{H}_{D_2}(X))=(X-j)$. Otherwise, there exists $j \in \mathbb{F}_{p^2} \backslash \mathbb{F}_p$ such that $\gcd (\overline{H}_{D_1}(X),\overline{H}_{D_2}(X))=(X-j)(X-j^p)$.
\end{theorem}
\begin{proof}
Let $(a,b, a')$ be a primitive reduced form in the genus class $\Lambda (q)$.

If $(a,b, a')$ has order $\le 2$, then $\gcd(\overline{H}_{D_1}(X),\overline{H}_{D_2}(X))=(X-j)$ with $j \in \mathbb{F}_p$. Moreover, we can find $x, y \in \mathbb{Z}$ such that $x^2-D_1y^2=4c$ (Appendix A).

If the order of $(a,b,a')$ is bigger than $2$ and $|D_1| \le 4$, then we have $-D_1=a=3$ and $c \equiv 1 \pmod 3$. It is easy to show that there exist $x, y \in \mathbb{Z}$ such that $x^2+3y^2=4c$ and $\gcd(\overline{H}_{D_1}(X),\overline{H}_{D_2}(X))=X$. In the following, we assume that the order of $(a,b,a')$ is bigger than $2$ and $a>4$. It is easy to show $c \nmid a$.

There exists a prime $q$ satisfying $(\ref{e1})$, which can be represented by $(a,b,a')$. Moreover, there exist a supersingular elliptic curve $E$ and a cyclic subgroup $G$ of $E$ of order $c$ such that $\text{End}(E,G)$ and $\text{End}(E^p,G^p)$ are isomorphic to $\mathcal{O}_c(q,r)$. Because $a$ and $a'$ are the first two successive minima of $\mathcal{O}_c(q,r)$, we have $(X-j(E)) \mid \gcd(\overline{H}_{-a}(X),\overline{H}_{-a'}(X))$.
By Proposition \ref{p5}, we have $\gcd(\overline{H}_{-a}(X),\overline{H}_{-a'}(X))=(X-j(E))(X-j(E^p))$ if $j(E) \in \mathbb{F}_{p^2} \backslash \mathbb{F}_p$.

Because $a<4\sqrt{cp/3}<p$, we have $(X-j(E)) \parallel \overline{H}_{-a}(X)$. If $j(E) \in \mathbb{F}_p$, then there exists an endomorphism $\phi: E \to E$ with $\deg(\phi)=c$. We can assume $a \equiv 3 \pmod 4$ and $a' \equiv 3 \pmod 4$, and other cases are similar.
Notice that $a$ and $a'$ are the first two successive minima of $\mathcal{O}_c(q,r)$, so there exist two elements $\gamma, \gamma' \in \text{End}(E)$ with $\text{Nrd}(\gamma)=a$, $\text{Nrd}(\gamma')=a$ and $\text{Trd}(\gamma)=\text{Trd}(\gamma')=0$. Moreover, $\mathbb{Z}[\frac{1+\gamma}{2}] \subseteq \text{End}(E)$. Since $c < 3p/16$, we have $c < \sqrt{3cp}/4  < a'/4$. So there exist $x', y' \in \mathbb{Z}$, such that $\text{Nrd}(x'+y'\frac{1+\gamma}{2})=c$. Let $x=2x'+y'$, $y=y'$, we have $x^2+ay^2=4c$. It is easy to show $\gcd(x,y)=1$.

On the contrary, we can assume that there exist $x, y \in \mathbb{Z}$ such that $x^2+ay^2=4c$. Without loss of generality, we also assume $a \equiv 3 \pmod 4$. The equation $b^2-4aa'=-16cp=-4p(x^2+ay^2)$ implies that $(2b)^2-4a(4(a'-py^2))=-16x^2p$. So we get a primitive form $f(X,Y)=aX^2+2bXY+4(a'-py^2)Y^2$ with discriminant $-16x^2p$.

The form $f(X,Y)$ can be derived from a primitive form $g(X',Y')$ with discriminant $-16p$. Without loss of generality, we assume $g(X',Y')=qX'^2+4rX'Y'+\frac{4(r^2+p)}{q}Y'^2$. We can find $h,i,j,k \in \mathbb{Z}$ with $hk-ij= \pm x$ such that $X=h X'+i Y'$, $Y=j X'+k Y'$, $a=g(h,j)$ and $4(a'-py^2)=g(i,k)$.

Assume that the supersingular elliptic curve $E$ defined over $\mathbb{F}_p$ satisfying $\text{End}(E) \cong  \mathcal{O}(q,r)$. Notice that $a=g(h,j)$. Since $ 2 \nmid a$, we have $2 \nmid h$. If there exists an odd prime $(c \neq )p_i \mid \gcd (h,j)$, then $p_i \mid a$ and $p_i \mid x$. Moreover, $p_i \mid x^2+ay^2(=4c)$. This is a contradiction. So $\gcd (h,j)=1$. It follows that $j(E)$ is a $\mathbb{F}_p$-root of $\overline{H}_{-a}(X)$.

Consider the following equation:
\begin{align}\label{e7}
p y^2 + q \left( X+\frac{2Zr}{q} \right)^2 + pq \left( y+ \frac{2Z}{q} \right)^2 = a'.
\end{align}
Let $X'=-4pX+4rp$ and $Z'=-4pZ-2pq$. We get the equation
\begin{align}\label{e8}
qX'^2+4rX'Z'+\frac{4(r^2+p)}{q}Z'^2=16p^2(a'-py^2).
\end{align}
By assumption, $4(a'-py^2)$ can be represented by $(q,4r, \frac{4(r^2+p)}{q})$. Because the prime $p$ is ramified in the field $\mathbb{Q}(\sqrt{-p})$, the solutions of (\ref{e8}) are $\pm(2pi,2pk)$. It easy to show $2 \mid i$. If $2 \mid k$, then $2 \mid x$ which implies $2 \mid y$ since $ 2 \nmid a$. This is a contradiction, since $\gcd(x,y)=1$. It follows $2 \nmid k$. Moreover, we have $x \equiv hk \equiv 1 \pmod 2$. We also have $2 \nmid y$. The equation (\ref{e7}) has solutions
$$\pm \left(\frac{2ry-i}{2},y,-\frac{k+qy}{2} \right) \ \text{and} \ \pm \left( \frac{i+ 2ry}{2},y,\frac{k-qy}{2} \right).$$

Notice that $ \gcd \left(\frac{2ry-i}{2},y,-\frac{k+qy}{2} \right)=\gcd \left(\frac{i+ 2ry}{2},y,\frac{k-qy}{2} \right)= 1$. It follows that $j(E)$ is also a root of $\overline{H}_{-a'}(X)$ and $(X-j(E))^2 \mid \overline{H}_{-a'}(X)$.

The proof is similar if $(a,b,a')$ is a reduced form in the genus class $\Lambda'(q)$.

\end{proof}

\begin{remark}
Let $(a,b,a')$ be a reduced form in the genus class $\Lambda(q)$ with order $>2$.
Assume that there exist $x,y \in \mathbb{Z}$ such that $x^2+ay^2=4c$. One can show that the number of solutions of $X^2 \equiv -16p \pmod {4a}$ is equal to the number of solutions of $X^2 \equiv -16cp \pmod {4a}$ with $-a < X \le a$. It follows that the number of reduced forms $(a,*,*)$ with discriminant $-16cp$ is the same as the number of $\mathbb{F}_p$-roots of $\overline{H}_{-a}(X)$. Similar result holds if $(a,b,a')$ is a reduced form in the genus class $\Lambda'(q)$.

\end{remark}

As we see, the fiber above $\mathcal{O}_c(q,r)$ or $\mathcal{O}'_c(q,r')$ along the map $\mathcal{O}(\cdot,\cdot)$ contains precisely the equivalence classes $(E,G)$ and $(E/G,\hat{G})$.
Considering the number of pairs $(E,G)$ where $E$ is a $\mathbb{Z}[\sqrt{-cp}]$-oriented supersingular elliptic curve, we have the following theorem.

\begin{theorem}\label{zt4}
Let $c$ be a prime with $c < 3p/16$.
If $E$ is a $\mathbb{Z}[\sqrt{-cp}]$-oriented (resp. $\mathbb{Z}[\frac{1+\sqrt{-cp}}{2}]$-oriented) supersingular elliptic curve, then there exists a $c$-isogeny $\phi: E \to E^p$ with $\ker(\phi)=G$ and $\operatorname{End}(E,G)$ is isomorphic to some $\mathcal{O}_c(q,r)$ (resp. $\mathcal{O}'_c(q,r')$).
\end{theorem}

\begin{proof}
The supersingular elliptic curve $E$ is $\mathbb{Z}[\sqrt{-cp}]$-oriented or $\mathbb{Z}[\frac{1+\sqrt{-cp}}{2}]$-oriented if and only if there exists a $c$-isogeny $\phi:E \to E^p$ with $\ker(\phi)=G$ (see \cite[Lemma 6]{MR2496385}).

As we know, each Eichler order $\mathcal{O}_c(q,r)$ (resp. $\mathcal{O}'_c(q,r')$) along the map $\mathcal{O}(\cdot,\cdot)$ contains the equivalence classes $(E,G)$ and $(E^p,G^p)$, where $E$ is $\mathbb{Z}[\sqrt{-cp}]$-oriented (resp. $\mathbb{Z}[\frac{1+\sqrt{-cp}}{2}]$-oriented). We will show that the number of isomorphic classes of $\mathcal{O}_c(q,r)$'s coincides with the number of $\mathbb{Z}[\sqrt{-cp}]$-oriented supersingular elliptic curves.

The number of equivalent classes of $\mathcal{O}_c(q,r)$ (resp. $\mathcal{O}'_c(q,r')$) is $h(-4cp)/4+\eta/2$ (resp. $h(-cp)/4+\eta'/2$) by Theorem \ref{zt2}. If an Eichler order $\mathcal{O}_c(q,r)$ (resp. $\mathcal{O}'_c(q,r')$) can represent a quadratic form $(a,b,a')$ in the genus class $\Lambda(q)$ (resp. $\Lambda'(q)$), we divide into two cases.

If the form $(a,b,a')$ has order $\le 2$, then we have $\gcd(\overline{H}_{-a}(X),\overline{H}_{-a'}(X))=(X-j)$ with $j \in \mathbb{F}_p$. Moreover, there exist $x,y \in \mathbb{Z}$ such that $x^2+ay^2=4c$ or $x^2+ay^2=c$. It follows that there are two elements $\lambda, \overline{\lambda} \in \text{End}(E(j))$ with $\text{Nrd}(\lambda)=\text{Nrd}(\overline{\lambda})=c$. Denote $G=\ker(\lambda)$ and $\overline{G}=\ker(\overline{\lambda})$.

If $c \mid a$, then $c$ is ramified in $\mathbb{Q}(\sqrt{-a})$. There exists $\mu \in \text{Aut}(E(j))$ such that $\overline{\lambda}=\mu \lambda$. We have $\text{End}(E(j),G) \cong \text{End}(E(j^p),G^p) \cong \mathcal{O}_c(q,r)$. If $c \nmid a$, then we have $a=1$ or $4$. As we know, $\mathcal{O}_c(q_1,r_1) \cong \mathcal{O}'_c(q_2,r_2)$ if and only if $\mathcal{O}_c(q_1,r_1)$ (resp. $\mathcal{O}'_c(q_2,r_2)$) can represent the form $(4,0,cp)$ (resp. $(1,1,\frac{1+cp}{4})$). These two forms correspond to the same $j$-invariant $1728$ with two different order-$c$ subgroups since $c$ splits in $\mathbb{Q}(\sqrt{-1})$. In conclusion, each reduced form in the genus class $\Lambda(q)$ or $\Lambda'(q)$ with order $\le 2$ corresponds to one supersingular elliptic curve with a subgroup of order $c$.

Assume that the order of the reduced form $(a,b,a')$ is bigger than $2$. We have $c \nmid a$. If $\gcd (\overline{H}_{D_1}(X),\overline{H}_{D_2}(X))=(X-j)$ with $j\in \mathbb{F}_p$, then $\text{End}(E(j),G) \not \cong \text{End}(E(j^p),G^p)$. If $\gcd (\overline{H}_{D_1}(X),\overline{H}_{D_2}(X))=(X-j)(X-j^p)$, then $\text{End}(E(j),G)$ and $\text{End}(E(j^p),G^p)$ are not isomorphic obviously. Moreover, $\mathcal{O}_c(q_1,r_1)$ is not isomorphic to $\mathcal{O}'_c(q_2,r_2)$, and each reduced form corresponds to $\text{End}(E(j),G)$ and $\text{End}(E(j^p),G^p)$.

By simple computation, these $h(-4cp)/4+\eta/2$ (resp. $h(-cp)/4+\eta'/2$) many isomorphism classes of Eichler orders $\mathcal{O}_c(q,r)$ (resp. $\mathcal{O}'_c(q,r')$) correspond to $h(-4cp)/2$ (resp. $h(-cp)/2$) many isomorphic classes of $\mathbb{Z}[\sqrt{-cp}]$-oriented (resp. $\mathbb{Z}[\frac{1+\sqrt{-cp}}{2}]$-oriented) supersingular elliptic curves $(E,G)$'s.

Considering quadratic twists of these curves (see \cite[Chapter 2]{chenu2021higher}), there are $h(-4cp)$ (resp. $h(-cp)$) many $\mathbb{Z}[\sqrt{-cp}]$-oriented ($\mathbb{Z}[\frac{1+\sqrt{-cp}}{2}]$-oriented) supersingular elliptic curves $(E,G)$'s. By Proposition \ref{p22}, this coincides with the cardinality of the set $SS_O^{pr}(p)$ where $O=\mathbb{Z}[-cp]$ (resp. $O=\mathbb{Z}[\frac{1+\sqrt{-cp}}{2}]$). It follows that $\operatorname{End}(E,G)$ is isomorphic to some $\mathcal{O}_c(q,r)$ if $E$ is $\mathbb{Z}[\sqrt{-cp}]$-oriented (resp. $\mathbb{Z}[\frac{1+\sqrt{-cp}}{2}]$-oriented).

\end{proof}

Note that different Eichler orders can correspond to the same $j$-invariant but with different $G$'s.

\begin{example}\label{example1}
Let $p=101$ and $c=3$. Denote $\mathbb{F}_{p^2}=\mathbb{F}_p(\alpha)$ with $\alpha^2+2=0$. There exist $5$ reduced forms in the genus class $\Lambda (q)$, and only one form has order $\le 2$.
We give these forms and the corresponding supersingular elliptic curves in the following table.
\begin{table}[H]
   \centering
    \caption{Reduced forms in $\Lambda'(q)$ and the corresponding supersingular elliptic curves}
    \footnotesize
   \begin{tabular}{lllll}
	\toprule
	Eichler order  &quadratic form  &$j$-invariant &elliptic curve   &cyclic group \\
	\midrule
	$\mathcal{O}'_3(1619,1215)$ &  $(3,\ 3,26)$ & $66$ & $E_1:y^2=x^3+8x+85$ & $\langle (50,12) \rangle$ \\
 \midrule
	\multirow{2}{*}{$\mathcal{O}'_3(11,7)$} &\multirow{2}{*}{$(8,\pm 7,11)$} &   $10\alpha+37$ & $E_2:y^2=x^3+(22\alpha+72)x+93\alpha+37$  & $\langle (49,50 \alpha+41) \rangle$\\
     &  & $91\alpha+37$ & $E_3:y^2=x^3+(79\alpha+72)x+8\alpha+37$  & $\langle (49,51 \alpha+41)\rangle$\\
     \midrule
	\multirow{2}{*}{$\mathcal{O}'_3(59,13)$} &\multirow{2}{*}{$(2,\pm 1, 38)$}
    &  $21$ & $E_4:y^2=x^3+(45\alpha+28)x+23\alpha+90$  & $\langle (8,91 \alpha+97) \rangle$   \\
     &  & $21$ & $E_5:y^2=x^3+(56\alpha+28)x+78\alpha+90$  & $\langle (8,10 \alpha+97) \rangle$ \\
    \bottomrule
\end{tabular}
 \end{table}
Notice that $E_3=E_2^p$ and $G_3=G_2^p$ with $G_3=\langle (49,51 \alpha+41)\rangle$ and $G_2=\langle (49,50 \alpha+41) \rangle$, but $(E_3,G_3)$ is not isomorphic to $(E_2,G_2)$.

There exist $5$ reduced forms in the genus class $\Lambda' (q)$, and only one form has order $\le 2$. We give these forms and the corresponding supersingular elliptic curves in the following table.
\begin{table}[H]
   \centering
    \caption{Reduced forms in $\Lambda(q)$ and the corresponding supersingular elliptic curves}
    \footnotesize
   \begin{tabular}{lllll}
	\toprule
	Eichler order  &quadratic form  &$j$-invariant &elliptic curve   &cyclic group \\
	\midrule
	$\mathcal{O}_3(1619,1215)$ & $(3,0,404)$ & $0$ & $E_6:y^2=x^3+82$ & $\langle (0,48) \rangle$ \\
\midrule
	\multirow{2}{*}{$\mathcal{O}_3(11,7)$} &\multirow{2}{*}{$(11,\pm 6,111)$}
    &  $57$ & $E_7:y^2=x^3+(48\alpha+29)x+55\alpha+43$ & $\langle (86,13 \alpha+56) \rangle$  \\
    &  & $57$ & $E_8:y^2=x^3+(53\alpha+29)x+46\alpha+43$ & $\langle (86,88 \alpha+56) \rangle$ \\
    \midrule
	\multirow{2}{*}{$\mathcal{O}_3(59,13)$} &\multirow{2}{*}{$(32,12,39)$}
    &  $91\alpha+37$ & $E_9:y^2=x^3+(76\alpha+48)x+93\alpha+50$ & $\langle (70,48 \alpha+89) \rangle$ \\
     &  & $10\alpha+37$ & $E_{10}:y^2=x^3+(25\alpha+48)x+8\alpha+50$ & $\langle (70,53 \alpha+89) \rangle$\\
    \bottomrule
\end{tabular}
 \end{table}
\small
Factoring Hilbert class polynomials $H_{-cp}(X)$ and $H_{-4cp}(X)$ in $\mathbb{F}_p$, we also get these $j$-invariants.
$$
\begin{array}{lcl}
  H_{-303}(X) & = & (X+35)^2(X+80)^4(X^2+27X+54)^2 \pmod {101}, \\
  H_{-1212}(X) & = & X^2(X+44)^4(X^2+27X+54)^2 \pmod {101}.
\end{array}$$
Moreover, we have $h(-303)=h(-4\times303)=10$ which is double number of curves in Table 3 and 4. Note that we do not give the twist curves.
%
\end{example}

\section{Oriented isogenies and binary quadratic forms}
We have established the correspondence between Eichler orders and oriented supersingular elliptic curves. In this section, we provide explicit generators for ideals corresponding to oriented isogenies and study the actions of oriented isogenies between oriented supersingular elliptic curves.

Denote the imaginary quadratic field $K=\mathbb{Q}(\sqrt{-cp})$. Let $O$ be an order of $K$ and $O_K$ be its ring of integers. All orientations are primitive in this section.

Let $E_1$, $E_2$ be $O$-oriented supersingular elliptic curves. Assume that $\varphi : E_1 \to E_2$ is an $O$-oriented isogeny. Without loss of generality, we can restrict to the case when $\text{deg}(\varphi)=\ell \neq c$ is a prime, since every isogeny can be written as compositions of isogenies of prime degree. Moreover, the prime $\ell$ splits in $O_K$.

\textbf{Case 1:} $ E_1$ and $ E_2$ are $\mathbb{Z}[\frac{1+\sqrt{-cp}}{2}]$-oriented.

In this case, we have $cp \equiv 3 \pmod 4$. We write $\mathcal{O}_c'(q,r)$ with $r^2 + cp \equiv 0 \pmod {4q}$. Note that $\mathcal{O}_c'(q,r)$ and $ \mathcal{O}_c'(q,-r)$ can represent different binary quadratic forms, even though they are isomorphic as quaternion orders. If $\text{End} (E_1,G_1) \cong  \mathcal{O}_c'(q_1,r_1)$, we can assume that $(E_1,G_1)$ corresponds to the form $(q_1,r_1, \frac{r_1^{2}+cp}{4q_1})$. Since $\ell$ splits in $O_K$ and $cp \equiv 3 \pmod 4$, there exists an integer $-\ell<b\le \ell$ satisfying $b^2 \equiv -cp \pmod {4\ell}$. In this case, $\ell$ has prime decomposition $\ell O_K=\mathfrak{l}\bar{\mathfrak{l}}$, where $\mathfrak{l}=[\ell, (-b+\sqrt{-cp})/2]$ (see \cite[Theorem 6.15]{MR1012948}).

\begin{proposition}
Assume that $cp \equiv 3 \pmod 4$. Put $\mathfrak{q}_1=\mathbb{Z} q_1+ \mathbb{Z}(r_1+\sqrt{-cp})/2$. Let $\mathfrak{l}$ be an integral ideal of $O_K$. We have an isomorphism: $\mathfrak{l}^{-1} \mathcal{O}'(q_1,r_1) \mathfrak{l} \cong \mathcal{O}'(q_2,r_2)$ for some $q_2$ satisfying $(\ref{e1})$ which
is the norm of a prime ideal $\mathfrak{q}_2 \in \bar{\mathfrak{q}}_1 \mathfrak{l}^{2} P(O_K)$.
\end{proposition}

\begin{proof}
Similar to the proof of Proposition 4.2 in \cite{MR683249}.

\end{proof}

Similar to Theorem 4.5 in \cite{xiao2022endomorphism}, we have the following theorem.
\begin{theorem}
Assume that $cp \equiv 3 \pmod 4$. Let $E_1$ be an $O_K$-oriented supersingular elliptic curve with $\operatorname{End} (E_1,G_1) \cong  \mathcal{O}_c'(q_1,r_1)$. If $\varphi : E_1 \to E_2$ is an $O_K$-oriented isogeny of degree $\ell$, then there exist an integer $b$ and a prime $q_2$ satisfying $(\ref{e1})$ which can be represented by the form $(q_1,r_1, \frac{r_1^{2}+cp}{4q_1})(\ell, b, \frac{b^2+p}{4\ell})^2 $ such that $\operatorname{End}(E_2,\varphi(G_1)) \cong \mathcal{O}'(q_2,r_2)$.
\end{theorem}

\textbf{Case 2:} $ E_1$ and $ E_2$ are $\mathbb{Z}[\sqrt{-cp}]$-oriented.

Similar relations also hold for Eichler order $\mathcal{O}_c(q,r)$. Let $E_1$ and $E_2$ be two $\mathbb{Z}[\sqrt{-cp}]$-oriented supersingular elliptic curves and $\varphi: E_1 \to E_2$ a $\mathbb{Z}[\sqrt{-cp}]$-oriented isogeny.

If $ \deg(\varphi)=\ell =2\neq c$, then the isogeny $\varphi$ corresponds to the ideal $[2, \sqrt{-cp}-1]$ in $\mathbb{Z}[\sqrt{-cp}]$ with $cp \equiv 1 \pmod 4$. Because the ideals $[2, \sqrt{-cp}-1]$ and $[\frac{cp+1}{2}, \sqrt{-cp}-1]$ are equivalent in the ideal class group $C(\mathbb{Z}[\sqrt{-cp}])$, there exists a $\mathbb{Z}[\sqrt{-cp}]$-oriented isogeny $\psi:E_1 \to E_2$ with kernel $\ker(\psi)=\{ P \in E_1[\frac{cp+1}{2}] : [\sqrt{-cp}-1]P = \infty \}$. Note that $\ker(\psi)=\{ P \in E_1[\frac{cp+1}{2}] : [2\sqrt{-cp}-2]P = \infty \}$ since $\frac{1+cp}{2}$ is odd. Moreover, the kernel ideal of $\psi$ can be represented by $[\frac{cp+1}{2}, 2\sqrt{-cp}-2]$ which corresponds to the form $(\frac{cp+1}{2},4, 8)\sim (8 ,-4, \frac{cp+1}{2})$.

If $\deg(\varphi)=\ell \neq c$ is an odd prime, then there exists an integer $-\ell<2b\le \ell$ satisfying $4b^2 \equiv -4cp \pmod {4\ell}$. We can assume that $\ker (\varphi)=\{ P \in E_1[\ell] : [\sqrt{-cp}-b]P=\infty \}$. Moreover, because $\{ P \in E_1[\ell] : [\sqrt{-cp}-b]P=\infty \}= \{ P \in E_1[\ell] : 2[\sqrt{-cp}-b]P=\infty \}$, the kernel ideal of $\varphi$ can be represented by $[\ell, 2(\sqrt{-cp}-b)]$ which corresponds to the form $(\ell, 4b, \frac{4b^2+4cp}{\ell})$.

The quadratic form $(8 ,-4, \frac{cp+1}{2})$ (or $(\ell, 4b, \frac{4b^2+4cp}{\ell})$) is primitive, so the corresponding ideal is proper.
Let $O=\mathbb{Z}[2\sqrt{-cp}]$ be an order in $K$ with discriminant $-16cp$. The prime $\ell \neq c$ has prime decomposition $\ell O=\mathfrak{l}\bar{\mathfrak{l}}$. We can write $\mathfrak{l}$ as following:
\begin{equation}\label{5}
\mathfrak{l}=\left \{
  \begin{array}{ll}
  (8 ,-4, \frac{cp+1}{2}) & \text{if} \ \ell=2 ; \\
  (\ell, 4b,\frac{4b^2+4cp}{\ell}) & \text{if} \ \ell \  \text{is odd .}
  \end{array} \right.
\end{equation}

\begin{proposition}
Let $O=\mathbb{Z}[2\sqrt{-cp}]$ be an order in $K$ with discriminant $-16cp$.
Let $\mathfrak{l}$ be an integral ideal of $O$ as in $(\ref{5})$. Put $\mathfrak{q}_1=\mathbb{Z} q_1+ \mathbb{Z}(2r_1+2\sqrt{-cp})$. We have an isomorphism: $\mathfrak{l}^{-1} \mathcal{O}_c(q_1,r_1) \mathfrak{l} \cong \mathcal{O}_c(q_2,r_2)$ for some $q_2$ satisfying $(\ref{e1})$ which is the norm of a prime ideal $\mathfrak{q}_2 \in \bar{\mathfrak{q}}_1 \mathfrak{l}^{2} P_{\mathbb{Z}}(2)$ (resp. $\mathfrak{q}_2 \in \bar{\mathfrak{q}}_1 \mathfrak{l}^{2} P_{\mathbb{Z}}(4)$) if $cp \equiv 1,2 \pmod 4$ (resp. $cp \equiv 3 \pmod 4$) respectively.
\end{proposition}
\begin{proof}
Note that $\mathfrak{l}$ is a proper ideal in $O$. The proof is similar to that of Proposition 4.7 in \cite{xiao2022endomorphism}.

\end{proof}

Similar to Theorem 4.9 in \cite{xiao2022endomorphism}, we have the following theorem.
\begin{theorem}
Let $E_1$ be a supersingular elliptic curve with $\operatorname{End} (E_1,G_1) \cong  \mathcal{O}_c(q_1,r_1)$. If $\varphi : E_1 \to E_2$ is a $\mathbb{Z}[\sqrt{-cp}]$-oriented isogeny of degree $\ell \neq c$, then there is a prime $q_2$ satisfying $(\ref{e1})$ which can be represented by the form $(q_1,4r_1, \frac{4r_1^{2}+4cp}{q_1})(\ell, 2b, \frac{b^2+4p}{\ell})^2 $ such that $\operatorname{End}(E_2,\varphi(G_1)) \cong \mathcal{O}_c(q_2,r_2)$.
\end{theorem}

\textbf{Case 3:} $E_1$ is $\mathbb{Z}[\frac{1+\sqrt{-cp}}{2}]$-oriented and $E_2$ is $\mathbb{Z}[\sqrt{-cp}]$-oriented.

In this case, we have $cp \equiv 3 \pmod 4$ and the oriented isogeny $\varphi : E_1 \to E_2$ is descending with $\deg(\varphi)=2$.

Let $q$ be a prime satisfying (\ref{e1}). As we know, $r^2+cp \equiv 0 \pmod {4q}$ implies $r^2+cp \equiv 0 \pmod {q}$. On the contrary, if the number $r$ satisfies the equation $x^2+cp \equiv 0 \pmod {q}$, then $r$ or $r+q$ satisfies the equation $x^2+cp \equiv 0 \pmod {4q}$. It follows that $\mathcal{O}_c(q,r)$ and $\mathcal{O}'_c(q,r)$ are Eichler orders with level $c$ for the same $r$ and $q$. The following theorem tells us the relations between $\mathcal{O}(q,r)$ and $\mathcal{O}'(q,r)$. Similar to the supersingular $\mathbb{F}_p$-isogeny graphs introduced in \cite{MR3451433}, the vertical $2$-isogeny does not change the choice of $q$.

\begin{theorem}
Let $cp \equiv 3 \pmod 4$ and $E_2$ be a $\mathbb{Z}[\sqrt{-cp}]$-oriented supersingular elliptic curve with $j(E_2) \neq 1728$. If $\operatorname{End}(E_2,G_2) \cong \mathcal{O}_c(q,r)$, then there exists an oriented isogeny $\varphi : E_2 \to E_1$ such that $\operatorname{End}(E_1,\varphi(G_2)) \cong \mathcal{O}'_c(q,r)$.
\end{theorem}

\begin{proof}
For $cp \equiv 3 \pmod 4$, if $E_2$ is a $\mathbb{Z}[\sqrt{-cp}]$-oriented supersingular elliptic curve, then there exists an oriented $2$-isogeny $\varphi : E_2 \to E_1$ where $E_1$ is $\mathbb{Z}[\frac{1+\sqrt{-cp}}{2}]$-oriented. By Theorem \ref{zt4}, we have $\operatorname{End}(E_1,\varphi(G_2)) \cong \mathcal{O}'_c(q',r')$ for some prime $q'$ satisfying (\ref{e1}). Moreover, the imaginary quadratic order $\mathbb{Z}[\frac{1+\sqrt{-q}}{2}]$ can be embedded into $\operatorname{End}(E_2,G_2) \cong \mathcal{O}_c(q,r)$, so the imaginary quadratic order $\mathbb{Z}[\sqrt{-q}]$ can be embedded into $\operatorname{End}(E_1,\varphi(G_2))$ (see \cite[Theorem 1]{MR4258517}). It follows that $\operatorname{End}(E_1,\varphi(G_2)) \cong \mathcal{O}'_c(q,r)$.
\end{proof}

\begin{remark}
If $j(E_2) = 1728$, we can assume $E_2:y^2=x^3+4x$. There exists a $2$-isogeny $\varphi:E_1 \to E_2$ with $E_1:y^2=x^3-x$ and $\varphi(x,y)=(\frac{x^2-1}{x}, \frac{x^2y+y}{x^2})$. Moreover, we have $j(E_1)=1728$ and $\operatorname{End}(E_1) \cong \operatorname{End}(E_2)$.
\end{remark}

Assume that $E_1$ is $\mathbb{Z}[\frac{1+\sqrt{-cp}}{2}]$-oriented with $\text{End}(E_1,G_1) \cong \mathcal{O}_c'(q,r)$ and $E_2$ is $\mathbb{Z}[\sqrt{-cp}]$-oriented with $\text{End}(E_2,G_2) \cong \mathcal{O}_c(q,r)$.

If $p \equiv 7 \pmod 8$, there is only one down $\mathbb{Z}[\sqrt{-cp}]$-oriented $2$-isogeny $\varphi: E_1 \to E_2$. If $\mathcal{O}_c'(q,r)$ can represent the form $(q,r,\frac{r^2+p}{4q})$ with $q$ satisfying $(\ref{e1})$, then $\mathcal{O}_c(q,r)$ can represent $(q,4r,\frac{4r^2+4p}{q})$.

If $p \equiv 3 \pmod 8$, there are three down $\mathbb{Z}[\sqrt{-cp}]$-oriented $2$-isogenies from $E_1$. Denote them by $\varphi_i: E_1 \to E_{2,i}$ with $i \in \{1,2,3\}$. If $\mathcal{O}_c'(q,r)$ can represent the form $(q,r,\frac{r^2+cp}{4q})$. There are three forms with discriminant $-16cp$ in the genus class $\Lambda(q)$ which can be derived from $(q,r,\frac{r^2+cp}{4q})$. Denote them by $(q_i,4r_i,\frac{4r_i^2+4cp}{q_i})$. Moreover, $\text{End}(E_{2,i},G_{2,i})$ can represent $(q_i,4r_i,\frac{4r_i^2+4cp}{q_i})$. On the contrary, suppose that $\text{End}(E_{2,i},G_{2,i})$ can represent the form $(q_i,4r_i,\frac{4r_i^2+4cp}{q_i})$ with $q_i$ satisfying $(\ref{e1})$. For every $i$, the form $(q_i,4r_i,\frac{4r_i^2+4cp}{q_i})$ can be derived from $(q_i,r_i,\frac{r_i^2+cp}{4q_i})$. It follows that the three forms $(q_i,r_i,\frac{r_i^2+cp}{4q_i}),i \in \{1,2,3\}$ are equivalent in $C(-cp)$ and $\text{End}(E_{1},G_{1})$ can represent each of these three forms.

\begin{example}
Notations are as Example \ref{example1}.
%
We compute the $2$-isogeny graph of these oriented supersingular elliptic curves in Example 1 as following.

We can use the form $(2,1,38)$ to represent the $2$-isogeny between $\mathbb{Z}[\frac{1+\sqrt{-cp}}{2}]$-oriented supersingular elliptic curves. It is easy to show that the results in this section hold for elliptic curves in this example.
\end{example}
$$
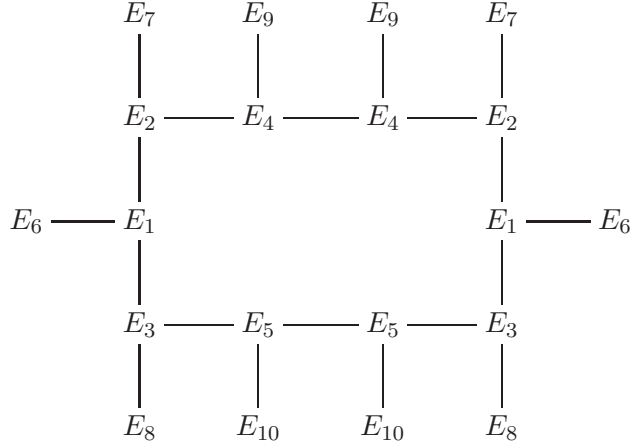
\xymatrix{
& & E_7 \ar@{-}[d]   &E_9 \ar@{-}[d] &E_{9} \ar@{-}[d] &E_7 \ar@{-}[d] &  \\
& & E_2 \ar@{-}[d] \ar@{-}[r]  &E_4 \ar@{-}[r] &E_4 \ar@{-}[r] &E_2 \ar@{-}[d] & \\
&E_6 \ar@{-}[r] & E_1 \ar@{-}[d] & & &E_1 \ar@{-}[d] & E_6 \ar@{-}[l] \\
& & E_3 \ar@{-}[r]  &E_5 \ar@{-}[r] &E_5 \ar@{-}[r] &E_3 &\\
& & E_8 \ar@{-}[u]  &E_{10} \ar@{-}[u] &E_{10} \ar@{-}[u] &E_8 \ar@{-}[u] & \\
}$$

\section{Conclusion}
Let $\left|\mathcal{S}_{c}\right|$ denote the set of pairs $(E, G)$, up to equivalence, where $E$ is a supersingular elliptic curve over $\overline{\mathbb{F}}_{p}$ and $G \subseteq E[c]$ is a subgroup of order $c$. The pairs in $\left|\mathcal{S}_{c} \right|$ are supersingular elliptic curves with level-$c$ structure. Arpin \cite{Arpin2203} has proved that the endomorphism ring $\text{End}(E,G)=\{ \alpha \in \text{End}(E): \alpha(G) \subseteq G \}$ is isomorphic to an Eichler order with level $c$ in $B_{p,\infty}$.

Let $c < 3p/16$ be a prime. If the supersingular elliptic curve $E$ is $\mathbb{Z}[\sqrt{-cp}]$-oriented, then we prove that $\text{End}(E,G)$ is isomorphic to some $\mathcal{O}_c(q,r)$ or $\mathcal{O}_c'(q,r)$. Moreover, if $q$ can be represented by a quadratic form $f=(a,b,a')$ in the genus class $\Lambda(q)$, then $\gcd (\overline{H}_{D_1}(X),\overline{H}_{D_2}(X))=(X-j)$ with $j \in \mathbb{F}_p$ or $\gcd (\overline{H}_{D_1}(X),\overline{H}_{D_2}(X))=(X-j)(X-j^p)$ whether there exist $x,y \in \mathbb{Z}$ such that $x^2+ay^2=4c$ or not. Similar results hold if $q$ can be represented by a quadratic form in the genus class $\Lambda'(q)$.

Note that there exists a $c$-isogeny $\phi: E\to E^p$ if the supersingular elliptic curve $E$ is $\mathbb{Z}[\sqrt{-cp}]$-oriented. If the prime $\ell \neq c$ satisfies $\left( \frac{-cp}{\ell} \right)=1$, then there exists an oriented $\ell$-isogeny $\varphi:E \to E'$. Equivalently, we can denote the isogeny $\varphi$ by a reduced form $g$. Assume that the endomorphism ring $\text{End}(E,G)$ of $E$ is isomorphic to $\mathcal{O}_c(q,r)$, where $q$ can be represented by a form $f$ with discriminant $-16cp$. Then the endomorphism ring $\text{End}(E',\varphi(G))$ of $E'$ is isomorphic to $\mathcal{O}_c(q',r')$, where $q'$ can be represented by the form $fg^2$. This shows an explicit relation between isogenies of $\mathbb{Z}[\sqrt{-cp}]$-oriented supersingular elliptic curves and the endomorphism rings of them.

If $c=1$, then the $\mathbb{Z}[\sqrt{-p}]$-oriented supersingular elliptic curves are exactly these curves over $\mathbb{F}_p$. Let $s$ be a positive integer. Let $\Omega_p(s)$ be the set of all $j$-invariants of $\mathbb{Z}[\sqrt{-cp}]$-oriented supersingular elliptic curves defined over $\mathbb{F}_{p^2} \backslash \mathbb{F}_p$, where $c$ runs all primes less than $s$. It is our future work to find the minimal $s$ such that $\Omega_p(s)$ contains all supersingular $j$-invariants defined over $\mathbb{F}_{p^2}\backslash \mathbb{F}_p$.

\section*{Acknowledgements}

The work is supported by the National Key Research and Development Program of China under Grant No. 2022YFA1004900, the National Natural Science Foundation of China under Grant No. 12201637  and No. 62202475, and the Innovation Program for Quantum Science and Technology under Grant No. 2021ZD0302902.

The authors would like to thank the anonymous reviewers for their valuable suggestions and comments on this paper.

\appendix
\section{The form $(a,b,a')$ with order $\le 2$}\label{Appendix}
\renewcommand{\appendixname}{}

We show $\gcd(\overline{H}_{D_1}(X),\overline{H}_{D_2}(X))=(X-j)$ with $j \in \mathbb{F}_p$, if the form $(a,b,a')$ is in the genus class $\Lambda(q)$ or $\Lambda'(q)$ with order $\le 2$. Let $c < 3p/16$ be a prime.

\textbf{Case 1}: $(1,1,\frac{1+cp}{4})$ ($p \equiv 3 \pmod 4$, $c \equiv 1 \pmod 4$)

If $p \equiv 3 \pmod 4$ and $c \equiv 1 \pmod 4$, then the form $(1,1,\frac{1+cp}{4})$ corresponds to $j=1728$ since $H_{-4}(X)=X-1728$ and $p \mid J(-4,-cp-1)$ by Proposition \ref{t5} and \ref{t6}.

\textbf{Case 2}: $(c,c,\frac{c+p}{4})$ $(cp \equiv 3 \pmod 4)$

If the form $(c,c,\frac{c+p}{4})$ is in the genus class $\Lambda'(q)$, then we have $\left( \frac{-p}{c} \right)=1$ and the equation $x^2 \equiv -16p \pmod {16c}$ is solvable. By \cite[Lemma 3.1]{xiao2022endomorphism}, the Hilbert class polynomial $H_{-4c}(X)$ mod $p$ has $\mathbb{F}_p$-roots.

\begin{prop}\label{pA.1}
Assume $\left( \frac{-p}{c} \right)=1$. If $p \equiv 1 \pmod 4$ and $c \equiv 3 \pmod 4$ with $4c < p$, then $H_{-4c}(X)$ has only one $\mathbb{F}_p$-root. If $p \equiv 3 \pmod 4$ and $c \equiv 1 \pmod 4$ with $4c < p$, then $H_{-4c}(X)$ has two $\mathbb{F}_p$-roots $j$ and $j'$. Moreover, there exists a $2$-isogeny from $E(j)$ to $E(j')$.
\end{prop}
\begin{proof}
The number of $\mathbb{F}_p$-roots of $H_{-4c}(X)$ can be given by \cite[Theorem 1.1]{MR4432517}. Note that \cite[Theorem 1.1]{MR4432517} also holds when $4c < p$.

If $p \equiv 3 \pmod 4$ and $c \equiv 1 \pmod 4$, then the equation $x^2 \equiv -p \pmod {4c}$ is solvable. Let $b$ be one solution. We have that $4(c+b)$ satisfies the equation $x^2 \equiv -16p \pmod {16c}$. Assume that the form $(c,b,\frac{b^2+p}{4c})$ (resp. $(4c,4(c+b),c+2b+\frac{b^2+p}{c})$) can be reduced to the form $f$ (resp. $f'$). Notice that $4 \mid \frac{b^2+p}{c}$, so the form $(4c,4(c+b),c+2b+\frac{b^2+p}{c})$ is primitive. Since the form $(4c,4(c+b),c+2b+\frac{b^2+p}{c})$ with discriminant $-16p$ is derived from $(c,b,\frac{b^2+p}{4c})$, the form $f'$ is derived from $f$.  Let $j, j'$ be the $\mathbb{F}_p$-roots of $H_{-4c}(X)$ mod $p$. If the curve $E(j)$ corresponds to the form $f$, then the curve $E(j')$ corresponds to the form $f'$ and there exists a $2$-isogeny $E(j)$ to $E(j')$.
\end{proof}

\begin{remark}
If $b$ is a solution of the equation $x^2 \equiv -p \pmod {4c}$ with $-2c < b \le 2c$, then $\pm b \pm 2c$ are all the solutions of this equation. These solutions correspond to two reduced forms $f$ and $f^{-1}$ and one supersingular $j$-invariant.
\end{remark}

Let $j \in \mathbb{F}_p$ be a root of $H_{-4c}(X)$. We can assume that the endomorphism ring of $E(j)$ is isomorphic to a maximal order $\mathcal{O}(q,r)$. We want to show that the imaginary quadratic order with discriminant $-(c+p)$ can be embedded into $\mathcal{O}(q,r)$.

Suppose there exists a $\gamma=w+x \frac{1+\beta}{2}+y \frac{\alpha'(1+\beta)}{2}+z \frac{(r+\alpha')\beta}{q} \in \mathcal{O}(q,r)$ satisfying:
$$\operatorname{Trd}(\gamma)=D \quad \text{and} \quad \operatorname{Nrd}(\gamma)= \frac{D^2-D}{4}.$$
These equations are equivalent to
$$2w + x =D$$ and
$$p y^2 + q \left( x+\frac{2zr}{q} \right)^2 + pq \left( y+ \frac{2z}{q} \right)^2 = -D.$$

Put $D=-(c+p)$ and $y= 1$. We have the following equation
\begin{align}\label{e4}
qx^2+4rxz+\frac{4(r^2+p)}{q}z^2+4pz=c-pq.
\end{align}

Let $x'=-4px+4rp$ and $z'=-4pz-2pq$. We get the equation
\begin{align}\label{e5}
qx'^2+4rx'z'+\frac{4(r^2+p)}{q}z'^2=16cp^2.
\end{align}
Since $O_{-4c}$ can be embedded into $\mathcal{O}(q,r)$, it follows that the equation $qx^2+4rxz+\frac{4(r^2+p)}{q}z^2=4c$ is solvable. Denote its solutions by $(m,n)$ and $(-m,-n)$. Obviously, we have $2 \mid m$. We claim $2 \nmid n$. If $2 \mid n$, then the equation $qx^2+4rxz+\frac{4(r^2+p)}{q}z^2=c$ is solvable. It follows that $j$ is a common root of $H_{-4c}(X)$ mod $p$ and $H_{-c}(X)$ mod $p$, which is impossible since $4c^2<p^2$. So we have $2 \mid m$ and $2 \nmid n$. Moreover, the prime $p$ is ramified in the field $\mathbb{Q}(\sqrt{-p})$, so the solutions of the equation (\ref{e5}) are $(2mp,2np)$ and $(-2mp,-2np)$. Then the equation (\ref{e4}) has solutions
$$\left( \frac{2r-m}{2}, \frac{-q-n}{2} \right) \text{and} \ \left(  \frac{2r+m}{2}, \frac{-q+n}{2} \right).$$

Let $y=-1$, we get a pair of solutions
$$\left( \frac{-2r+m}{2}, \frac{q+n}{2} \right) , \ \left(  \frac{-2r-m}{2}, \frac{q-n}{2} \right).$$

This method is from Gauss \cite[\S 216]{MR837656}. It follows that $O_{-(c+p)}$ can be embedded into $\mathcal{O}(q,r)$. Moreover, we have $(X-j)^2 \mid H_{-(c+p)}(X) \pmod p$.

\begin{remark}
If $j$ is the common root of $H_{-4c}(X)$ mod $p$ and $H_{-(c+p)}(X)$ mod $p$, then $(X-j)^2 \mid H_{-(c+p)}(X)$ mod $p$, because there are two elements $\gamma, \gamma' \in \operatorname{End}(E(j))$ with $\gamma \neq \pm  \gamma'$ such that $\operatorname{Trd}(\gamma)=\operatorname{Trd}(\gamma')=D$ and $\operatorname{Nrd}(\gamma)=\operatorname{Nrd}(\gamma')=\frac{D^2-D}{4}$.
\end{remark}

On other hand, suppose there exist $x,k \in \mathbb{Z}$ with $4c(c+p)-x^2=4kp>0$. We have $x \equiv \pm 2c \pmod {p}$. Since $c<3p/16$ and $x^2 < 4c(c+p)< p^2$, we have $x= \pm 2c$ and $k=c$. We can compute $v_p(J(-4c,-(c+p)))=2$. So $H_{-4c}(X)$ and $H_{-(c+p)}(X)$ have only one common $\mathbb{F}_p$-root.

If $p \equiv 1 \pmod 4$ and $c \equiv 3 \pmod 4$, then $H_{-4c}(X)$ has only one $\mathbb{F}_p$-root which is the common $\mathbb{F}_p$-root of $H_{-4c}(X)$ and $H_{-(c+p)}(X)$.

If $p \equiv 3 \pmod 4$ and $c \equiv 1 \pmod 4$, then $H_{-4c}(X)$ has two $\mathbb{F}_p$-roots. Denote by $j$ and $j'$. Assume $\text{End}(E(j)) \cong \mathcal{O}(q,r)$ and $\text{End}(E(j')) \cong \mathcal{O}'(q,r')$. Then $j$ is the common $\mathbb{F}_p$-root of $H_{-4c}(X)$ and $H_{-(c+p)}(X)$.

\textbf{Case 3}: $(4,0,cp)$

The form $(4,0,cp)$ is in the genus class $\Lambda(q)$ if and only if $p \equiv 3 \pmod 4$ and $c \equiv 1 \pmod 4$. In this case, the Hilbert class polynomial $H_{-4}(X)=X-1728$ has only one $\mathbb{F}_p$-root $1728$. The equation $4cp-x^2=4kp$ with $x\in \mathbb{Z}$, $k \in \mathbb{Z}_+$ is solvable if $x=0$. We have $p \mid J(-4,-cp)$, and Hilbert class polynomials $H_{-4}(X)$ and $H_{-cp}(X)$ have only one common $\mathbb{F}_p$-root.

\textbf{Case 4}: $(4c,0,p)$

The form $(4c,0,p)$ is in the genus class $\Lambda(q)$ if and only if $p \equiv 3 \pmod 4$ and $\left( \frac{c}{p}\right)=1$. If $p \equiv 3 \pmod 4$ and $c\equiv 1 \pmod 4$, then the Hilbert class polynomial $H_{-4c}(X)$ has two $\mathbb{F}_p$-roots and only one root $j$ satisfying $\mathbb{Z}[\frac{\sqrt{-p}+1}{2}] \subseteq \text{End}(E(j))$ by Proposition \ref{pA.1}. If $p \equiv 3 \pmod 4$ and $c\equiv 3 \pmod 4$, then the Hilbert class polynomial $H_{-4c}(X)$ has only one $\mathbb{F}_p$-root by Theorem 1.1 in \cite{MR4432517}. The equation $x^2 \equiv -p \pmod {4c}$ is solvable since $\left( \frac{-p}{c}\right)=\left( \frac{c}{p}\right)=1$. It follows that the Hilbert class polynomials $H_{-4c}(X)$ and $H_{-p}(X)$ has only one common $\mathbb{F}_p$-root.

\textbf{Case 5}: $(c,0,4p)$

The form $(c,0,4p)$ is in the genus class $\Lambda(q)$ if and only if $c \equiv 3 \pmod 4$ and $\left( \frac{-4p}{c}\right)=1$. Since the equation $x^2 \equiv -16p \pmod {4c}$ is solvable, by \cite[Lemma 3.4]{xiao2022endomorphism}, we have that the Hilbert class $H_{-c}(X)$ mod $p$ has only one $\mathbb{F}_p$-root $j$ and $\mathbb{Z}[\sqrt{-p}]$ can be embedded into $\text{End}(E(j))$ optimally. It follows that the Hilbert class polynomials $H_{-c}(X)$ and $H_{-4p}(X)$ has only one common $\mathbb{F}_p$-root.

\textbf{Case 6}: $(4,4,1+2p)$ and $(8,8,2+p)$

The form $(4,4,1+2p)$ (resp. $(8,8,2+p)$) is in the genus class $\Lambda(q)$ if and only if $p \equiv 3 \pmod 4$ (resp. $p \equiv 5 \pmod 8$). As case 3, Hilbert class polynomials $H_{-4}(X)$ (resp. $H_{-8}(X)$) and $H_{-(1+2p)}(X)$ (resp. $H_{-(2+p)}(X)$) have only one common $\mathbb{F}_p$-root respectively.

\end{document}